\documentclass{article}
\usepackage[colorlinks]{hyperref} 
\usepackage{latexsym, amssymb, amscd, amsfonts, amsthm, amsmath, graphicx,
mathabx, mathrsfs, tikz}
\usepackage{tikz-cd} 
\usepackage{cleveref}

\newcommand{\E}{{\mathbb{E}}}

\newtheorem{theorem}{Theorem}[section]

\newtheorem{lemma}[theorem]{Lemma}
\newtheorem{corollary}[theorem]{Corollary}
\newtheorem*{lemma*}{Lemma}
\newtheorem{proposition}[theorem]{Proposition}

\theoremstyle{remark}
\theoremstyle{definition}
\theoremstyle{corollary}

\newtheorem{definition}{Definition}[section]

\title{Exponential Absolute Minimizing extension and biased infinity Laplacian}
\author{Yang Chu}
\begin{document}

\maketitle

\begin{abstract}
     We study the variational structure of the biased infinity Laplacian by introducing a notion of the $\beta$\textit{-Exponential Absolute Minimizing Extension} ($\beta$--AM) on arbitrary length space, which absolutely minimizing the exponential slope
     $$
L^{\beta}_u (E) := \beta \sup_{x,y \in E} \frac{u(y) - e^{-\beta |x-y|} u(x)}{1- e^{-\beta |x-y|}}.
$$We also define the corresponding Exponential McShane-Whitney-type extension, and $\beta$-biased convexity, which equivalently characterize $\beta$-AM and may be of independent interest. These generalize the classical Absolute Minimizing Lipschitz Extension as a special case when $\beta = 0$. In Euclidean space with Euclidean norm, this corresponds to the Aronsson equation with Hamiltonian 
\[
H(u, \nabla u) = |\nabla u| + \beta u,
\]
equivalently viscosity solutions of $\Delta_{\infty}^{\beta} u = 0$. We show that $\beta$-AM arises as the continuum value of a biased tug-of-war game.
Analogous to the unbiased case, we derive various properties of this extension. As an application, we further show that the linear blow-up property holds for biased infinity harmonic functions.

\end{abstract}

\section{Introduction}

\subsection{Exponential Absolute Minimizing functions and main results}
The biased infinity Laplacian 
$$ 
\Delta_{\infty}^{\beta} u := \Delta_{\infty}^N u + \beta |\nabla u|
= \frac{\nabla u^T D^2 u \nabla u}{|\nabla u|^2} + \beta |\nabla u|
$$
and comparison with exponential cones are
introduced in \cite{peres2010biased} to characterize the continuum values of biased tug-of-war games on compact length spaces. We use $\Delta_{\infty}^N $ to denote the normalized 1-homogeneous infinity Laplacian. While the unbiased infinity Laplacian arises from calculus of variation in $L^{\infty}$ norm and more generally Absolutely Minimizing Lipschitz extension (AMLE), such variational structure is missing so far for the biased infinity Laplacian. We will show there is actually a rich variational structure in the biased case as well.

Without loss of generality we assume $\beta>0$ throughout the paper, as the reader can check the case for $\beta<0$ can be handled by fliping the sign of the boundary value and working with $-\beta$. We introduce a new notion of $\beta-$\textit{Exponential Absolute Minimizing extension}, on any length space (not necessarily compact) $X$, which generalizes the concept of Absolute Minimizing Lipschitz extension as a special case when $\beta \downarrow 0$. 

\begin{definition}[Exponential slope]\label{def_biased_slope}
    We define the biased $\beta$-slope of a function $u$ on metric space $E$ as:
$$
L^{\beta}_u (E) = \beta \sup_{x,y \in E} \frac{u(y) - e^{-\beta |x-y|} u(x)}{1- e^{-\beta |x-y|}} 
$$
Note that $L^{0} = \lim_{\beta \downarrow 0} L^{\beta} = \sup_{x,y \in E} \frac{u(y) -  u(x)}{|x-y|}$ reduce to the usual Lipschitz constant.

A minimal $\beta -$\textit{biased extension} on $X\setminus Y$ with given data $g$, i.e. a real valued function, on $Y$ with $\beta$-exponential constant $L^{\beta}$, is a function $u$ on $X$ such that the $\beta-$ exponential constant of $u$ on $X$ is at most $L^{\beta}$, i.e. $L^{\beta}_u (X) = L^{\beta}_u (Y) = L^{\beta}_g (Y)$.

The  \textbf{$\beta$ biased absolute minimal Lipschitz extension ($\beta$-AM)} is a minimal $\beta-$extension such that 
$L^{\beta}_u (U) =L^{\beta}_u (\partial U)$ for any open set $U$.

We shall suppress the notion of bias $\beta$ and write $L_u$, when there is no confusion.
\end{definition}

Roughly speaking, the Lipschitz extension interpolates the function such that it is linear in the direction of the gradient. Where in the $\beta$-extension, it interpolates the function exponentially, such that for $\beta >0$ (resp. $\beta <0$) it is concave (resp. convex) in the direction of gradient.

Here we define $\beta$-exponential cone and its comparison principle. Note that our $\beta$-exponential cone is equivalent to the one given in \cite{peres2010biased}, but the comparison principle is different from \cite{peres2010biased}. In particular, the comparison with exponential cone from above is only required for the positive cone, and from below is only required for the negative cone. The statement for comparison from above for both positive and negative is in fact not true as one can check in the proof. 
Also, we do not assume that the test open set $V$ to be bounded as in \cite{peres2010biased} which only deal with compact space. 
Also, the center $x_0$ of the cone is outside of the test set $V$, which is natural from the PDE point of view \cite{aronsson2004tour} and simplifies the proof. We shall show later that comparison with exponential cone from above is equivalent to viscosity sub-solution.
\begin{definition}[Exponential cones]\label{def_exp_cone}
    For $\beta >0$, given $A \geq 0$ and $x_0 \in X$ in a length space, we define

$$
\begin{aligned}
& C_{x_0}^{+}(x):=  A \frac{1}{\beta}\left(1-e^{-\beta |x-x_0|}\right) \geq 0 \\
& C_{x_0}^{-}(x):=  A \frac{1}{\beta} \left(1-e^{\beta |x-x_0|}\right) \leq 0
\end{aligned}
$$

For any $B,K \in \mathbb{R}$, the functions $C_{x_0}^{+}(x)+ K$ and $C_{x_0}^{-}(x) +K $ are called positive and negative $\beta$-exponential cones centered at $x_0$ with "slope" $A$. Note that when $\beta \to 0$, both positive and negative $\beta$-exponential cones reduces to the $A|x-x_0|+ B +K$ which corresponds to the usual cone function.
\end{definition}

\begin{definition}[Comparison with exponential cones]
    A function $u: X \backslash Y \longrightarrow \mathbb{R}$ satisfies comparison with $\beta$-exponential cones from above (CECA) if for any open set $V \subset X \backslash Y$,  $ x_0 \in X \setminus V$ and $B \in \mathbb{R}$, whenever $u(x) \leq C_{x_0}^{+}(x)+Be^{-\beta |x-x_0|} + K$ holds on $x \in \partial\left(V \backslash\left\{x_0\right\}\right)$, we also have $u \leq C_{x_0}^{+}+Be^{- \beta |x-x_0|} +K$ on the entire $V$. 
    
    Comparison with $\beta$-exponential cones from below (CECB) is defined analogously: in the same setting, whenever $u(x) \geq C_{x_0}^{-}(x)+Be^{\beta |x-x_0|} +K$ holds on $x \in \partial\left(V \backslash\left\{x_0\right\}\right)$, we also have $u \geq C_{x_0}^{-}+Be^{\beta |x-x_0|} +K $ on the entire $V$.  
    
    If both conditions hold, then $u$ satisfies comparison with $\beta$-exponential cones (CEC).

    In words, CECA means if the $u$ is upper bounded on the boundary of $V$ by a positive $\beta$-exponential cone with vertex not in $V$, it is also upper bounded on the entire $V$. CECB similarly means lower bound by the negative $\beta$-exponential cone.
\end{definition}

As we will show, the $\beta$-extension shares common features as the Lipschitz extension as thoroughly explained in \cite{aronsson2004tour}. Our presentation closely follows \cite{aronsson2004tour} and \cite{crandall2001optimal}, as the readers will find many similarities but very different behavior.

We need some definition before stating the main results. 

\begin{definition}[Strongly absolute minimizing and local Exponential slope]\label{def_local_slope}
    We define the local $\beta$-slope at $y$ as
    $$
         T^{\beta}_u(y) := \lim_{r \downarrow 0} L^{\beta}_u(B_r(y))
         = \inf \left\{{L}_u^{\beta}\left(B_r(y)\right): 0<r<\operatorname{dist}(y, Y)\right\}
         $$
    We shall suppress the superscript $\beta$ and write $T_u$ as there is no confusion.
     For $u \in C(U)$. Then $u$ is $\beta$ strongly absolutely minimizing ($u \in \beta\operatorname{-AMS}(U))$ if for every $V \subset \subset U$ and $v \in C(\bar{V})$ such that $u=v$ on $\partial V$, one has

$$
\sup _{x \in V} {~T}_u(x) \leq \sup _{x \in V} {~T}_v(x)
$$

\end{definition}

\begin{definition}[Exponential McShane-Whitney extension]\label{def_Mchshane_Whit}
    We define a special case of $\beta$-extension which is an analogy to the one of McShane \cite{mcshane1934extension} and Whitney \cite{Whitney34}. Thus We shall call this biased McShane-Whitney extension.
    Given any open set $U$, we define
    $$\begin{aligned} 
    & \Psi(g)(x):=\inf _{y \in Y}\left(  e^{-\beta |x-y|}g(y)
    + \frac{L_g(Y)}{\beta} (1-e^{-\beta |x-y|}))\right), \\ 
    & \Lambda(g)(x):=\sup _{y \in Y}
    \left( e^{\beta |x-y|}g(y)
    + \frac{L_g(Y)}{\beta} (1-e^{\beta |x-y|}))
    \right).
    \end{aligned}
    $$
    A function $u$ defined on $\Omega$ satisfies comparison with $ \Psi(u)$ from above if for any open set $U \subset \Omega$, $u \leq \Psi(u)$.
    Similarly $u$ satisfies comparison with $\Lambda (u)$ from below if for any open set $U \subset \Omega$, $u \geq \Lambda(u)$.

\end{definition}

Note that any unbiased infinity harmonic function $u$, defined on an open set $U$ and $y \in U$, is equivalently characterized by the folloiwing properties (see \cite{aronsson2004tour}):
\begin{align*}
g^{+}(r) &= \max_{|w - y| \le r} u(w) \quad \text{is convex for } 0 \le r < \mathrm{dist}(y, Y) \text{ and} \\
g^{-}(r) &= \min_{|w - y| \le r} u(w) \quad \text{is concave for } 0 \le r < \mathrm{dist}(y, Y).
\end{align*}

Motivated by this, we introduce the notion of $\beta$-convexity, associated with the $\beta$-AM, which may be of independent interest. Although the definition may appear somewhat intricate at first glance, its underlying intuition is straightforward: ordinary convexity asserts that
$$
f\left( (1-t) x_1+ t x_2 \right) \leq 
(1-t) f\left(x_1\right) 
+ t f\left(x_2\right)
$$

Note that the right hand side is the same as 
$$\E_{(1-t) x_1+ t x_2} f(B_\tau)$$
, where $B_s$ is a Brownian motion on the line segment from $x_1$ to $x_2$, and $\tau= \inf \{s>0: B_s = x_1 \text{ or } x_2 \}$. Now if we replace it with Brownian motion with drift $\beta/2$, where the drift is in the direction of $\arg \max_{x_1,x_2} f(x)$. I.e. when $\beta >0$, the drift is positive toward $arg \max_{x_1,x_2} f(x)$. When $\beta<0$, the drift is negative toward $arg \max_{x_1,x_2} f(x)$. One can easily check that the below definition (for $\beta$-convex) corresponds exactly to this generalization. This is not surprising, as the biased tug-of-war could be viewed as a game of controlling a Brownian motion with drift.

\begin{definition}[Biased convexity]
    For all $0 \leq t \leq 1$ and all $x_1, x_2 \in X$, we say a function $u$ on a vector space $X$ is $\beta$-convex if  :
$$
f\left((1-t) x_1+ t x_2\right) \leq 
\frac{ e^{-\beta t} - e^{-\beta} }{1-e^{-\beta}} f\left(x_1\right) 
+ \frac{ 1- e^{-\beta t} }{1-e^{-\beta}} f\left(x_2\right), \text{ if } f(x_1) < f(x_2)
$$
$$
f\left((1-t) x_1+ t x_2\right) \leq 
\frac{ e^{\beta t} - e^{\beta} }{1-e^{\beta}} f\left(x_1\right) 
+ \frac{ 1- e^{\beta t} }{1-e^{\beta}} f\left(x_2\right), \text{ if } f(x_1) \geq  f(x_2)
$$

Let $x_m = \arg \max_{x_1,x_2} f(x), x_M = \arg \max_{x_1,x_2} f(x), m = \min\{f\left(x_1\right),f(x_2)\}$, $M = \max \{f\left(x_1\right),f(x_2)\}$
In a single equation this gives,
\begin{equation}\label{beta_convex_original}
    f\left((1-t) x_m+ t x_M\right) \leq 
\frac{ e^{-\beta t} - e^{-\beta} }{1-e^{-\beta}} m 
+ \frac{ 1- e^{-\beta t} }{1-e^{-\beta}} M, 
\end{equation}

Equivalently, 

\begin{equation}\label{incr_beta_convex_m_M}
    f\left((1-t) x_1+t x_2\right) \leq \alpha(t) m+(1-\alpha(t)) M
\end{equation}

where $$\phi(t)=\frac{e^{-\beta t}-e^{-\beta}}{1-e^{-\beta}}, \quad \alpha(t)=\min \{\phi(t), \phi(1-t)\}=\phi(\max \{t, 1-t\})
$$

Similarly, we say a function is $\beta$-concave if
\begin{equation}\label{beta_concave_original}
    f\left((1-t) x_m+ t x_M\right) \geq 
\frac{ e^{\beta t} - e^{\beta} }{1-e^{\beta}} m 
+ \frac{ 1- e^{\beta t} }{1-e^{\beta}} M, 
\end{equation}
\end{definition}

When $\beta =0$, this reduces to the usual convexity/concavity. Also from definition, if $u$ is $\beta$-convex (resp. concave), $u$ is also $\beta^{\prime}$-convex (resp. concave) for any $\beta^{\prime} > \beta$. In particular, any convex function is $\beta$-convex for $\beta>0$. One can also check that $u$ is $\beta$-convex if and only if $-u$ is $\beta$-concave.

Now we're ready to state one main result of equivalent characterization of $\beta$-AM in the following theorem. 
Note that the assumption $L^{\beta}_g(Y) < \infty$ implies $\sup_{Y} g \leq \frac{L^{\beta}_g(Y)}{\beta} < \infty$, i.e. $g$ is bounded above, and $g$ is locally Lipschitz.

We first state everything in compact length space for completeness.

\begin{theorem}\label{main}
    Given a compact length space $X$, $Y \subset X$ and a function $g$ on $Y$ such that $L^{\beta}:= L^{\beta}_g(Y) < \infty$, then there exists an unique $\beta-$AMLE, and
    the following are equivalent: 
    \begin{enumerate}
        \item $u$ is an absolute minimal $\beta$-extension. \label{thm_ambe}
        \item $u$ satisfies comparison with $ \Psi(u)$ from above, and $\Lambda (u)$ from below. \label{thm_comparison_whitt}
        \item $u$ satisfies comparison with exponential cones. \label{thm_comparison_cones}
        \item $u$ is strongly absolute minimizing. \label{thm_stronglyAM}
        \item $u$ is the continuum value of the biased tug-of-war. \label{thm_tug_of_war}
        \medskip
        
        The next two conditions are also equivalent to the above and do not involve the test open set:
        \item $ \forall  y \in X\setminus Y $ and $0 \le r < \mathrm{dist}(y, Y)$, we have
\[
\begin{array}{ll}
(i) & g^{+}(r) = \displaystyle \max_{|w - y| \le r} u(w) \quad \textit{is $\beta$-convex, }  \textit{ and} \\[1.2em]
(ii) & g^{-}(r) = \displaystyle \min_{|w - y| \le r} u(w) \quad \textit{is $\beta$-concave}.
\end{array}
\] \label{thm_convexity}
        \item $ \forall  y \in X\setminus Y, $ and $|x-y|<r< \mathrm{dist}(y, Y)$ we have
        \[
\begin{array}{ll}
(i) & u(x) \leq  e^{-\beta |x-y|}u(y) + \max _{\{w:|w-y|=r\}}\left( 
\frac{u(w)- e^{-\beta r}u(y)}{1- e^{-\beta r}}
\right)  \left({1- e^{-\beta |x-y|}}\right) \textit{ and} \\[1.2em]
(ii) & u(x) \geq  e^{\beta |x-y|}u(y)+
\max _{\{w:|w-y|=r\}}\left( 
\frac{u(w)- e^{\beta r}u(y)}{1- e^{\beta r}}
\right)  \left({1- e^{\beta |x-y|}}\right)
\end{array}
\] \label{thm_slope_estimate}

    \bigskip

    Moreover, if $X$ is a compact subset of $\mathbb{R}^n$, the followings are also equivalent, together with the above.
        \item $u$ is a viscosity solution of 
        \begin{equation} \label{biased_infty_lap}
                    \Delta_{\infty}^N u + \beta |\nabla u|=0
        \end{equation} \label{thm_biased_infty_lap}
        \item $u$ minimizes the following $L^{\infty}$ functional 
        \begin{equation}\label{mini_functional}
        \| |\nabla u| + \beta u \|_{L^\infty},
        \end{equation}
        which could correspond \cref{biased_infty_lap} to be its Euler-Aronsson equation. \label{thm_variational}
        \item $u$ satisfies the following mean value property in the viscosity sense: 
        $$
        u(x) = \frac{1+ {\beta} \epsilon/2}{2} \sup_{y\in B_{\epsilon}(x)} u (y)
        + \frac{1- {\beta} \epsilon/2}{2} \inf_{y\in B_{\epsilon}(x)} u (y) + o(\epsilon^2)
        $$ \label{thm_mean_value}
    \end{enumerate}
\end{theorem}

In the general case, not all of the above theorem holds. 

\begin{theorem}\label{main_general}
    Given a length space $X$, $Y \subset X$ and a function $g$ on $Y$ such that $L^{\beta}:= L^{\beta}_g(Y) < \infty$ and additional $g$ is bounded below, then there exists an unique $\beta-$AMLE, and
    the following are equivalent: 
    \begin{enumerate}
        \item $u$ is an absolute minimal $\beta$-extension. \label{thm_ambe_gene}
        \item $u$ satisfies comparison with $ \Psi(u)$ from above, and $\Lambda (u)$ from below. \label{thm_comparison_whitt_gene}
        \item $u$ satisfies comparison with exponential cones. \label{thm_comparison_cones_gene}
        \item $u$ is strongly absolute minimizing. \label{thm_stronglyAM_gene}
        \item $u$ is the continuum value of the biased tug-of-war. \label{thm_tug_of_war_gene}
        \medskip
        
    \bigskip 
    
    Moreover, if $X$ is a subset of $\mathbb{R}^n$, the following is also equivalent together with the above.
        \item $u$ is the smallest viscosity solution of 
        \begin{equation} \label{biased_infty_lap}
                    \Delta_{\infty}^N u + \beta |\nabla u|=0
        \end{equation} \label{thm_biased_infty_lap_gene}
    \end{enumerate}
\end{theorem}

In the general case, the properties without test sets are no longer valid. Roughly speaking, those are local properties and do not imply the absolute minimizing property on unbounded test sets. This confirms a observation and resolves an open problem in \cite{peres2010biased} section 7.

The existence of $\beta$-AM in general setting is established through biased tug-of-war which we introduce in the next subsection.

\subsection{Biased tug-of-war on graphs}

1.2. Random turn games and values. We consider two-player, zero-sum randomturn games, which are defined by the following parameters: a set $X$ of states of the game, two directed transition graphs $E_{\mathrm{I}}, E_{\mathrm{II}}$ with vertex set $X$, a nonempty set $Y \subset X$ of terminal states (a.k.a. absorbing states), a terminal payoff function $g: Y \rightarrow \mathbb{R}$, a running payoff function $f: X \backslash Y \rightarrow \mathbb{R}$, and an initial state $x_0 \in X$.

Game play is as follows: a token is initially placed at position $x_0$. At the $k^{\text {th }}$ step of the game, a coin is tossed where Player I wins with probability $\frac{\rho}{\rho +1}$ and $\frac{1}{\rho +1}$ for Player II, and the player who wins the toss may move the token to any $x_k$ for which $\left(x_{k-1}, x_k\right)$ is a directed edge in her transition graph. The game ends the first time $x_k \in Y$, and player I's payoff is $g\left(x_k\right)+\sum_{i=0}^{k-1} f\left(x_i\right)$. Player I seeks to maximize this payoff, and since the game is zero-sum, player II seeks to minimize it. 

We will use the term tug-of-war (on the graph with edges $E$ ) to describe the game in which $E:=E_{\mathrm{I}}=E_{\mathrm{II}}$ (i.e., players have identical move options) and $E$ is undirected (i.e., all moves are reversible). We now define the value functions $u_I$ for player I, and $u_{I I}$ for player II. A strategy for a player is a way of choosing the player's next move as a function of all previous coin tosses and moves. Given two strategies, $S_I$ for player I and $S_{II}$ for player II, we define the payoff functions 
$F_{+}(S_I, S_{II})$ and $F_{-}(S_I, S_{II})$ as the expected payoff at the end of the game if the game terminates 
with probability one under these strategies. So, in this case they are equal. However, we let 
$F_{-}(S_I, S_{II}) = -\infty$ and $F_{+}(S_I, S_{II}) = \infty$ otherwise. 
We think of $u_I^{}$ as the minimum that player I can guarantee being paid; on the other hand, 
player II can avoid having to pay more than the value $u_{II}^{}$. More precisely,

\[
u_I^{} := \sup_{S_I} \inf_{S_{II}} F_{-}(S_I, S_{II}),
\]
\[
u_{II}^{} := \inf_{S_{II}} \sup_{S_I} F_{+}(S_I, S_{II}).
\]

Note that in the biased case, player I can always force the game to end almost surely by targeting a fixed point $y\in Y$. But player II may not be able to force the game to end almost surely, then $u_{\text {II }}=\infty$. If $X$ is compact, player II could indeed force the game to end almost surely as shown in \cite{peres2010biased}.

From the definition we see $u_{\mathrm{I}}(x) \leq u_{\mathrm{II}}(x)$. When $u_{\mathrm{I}}(x)=u_{\mathrm{II}}(x)$, we say that the game has a value, given by $u(x):=u_{\mathrm{I}}(x)=u_{\mathrm{II}}(x)$.

Our definition of value for player I penalizes player I severely for not forcing the game to terminate with probability one, awarding $-\infty$ in this case.

\begin{lemma}
    The function $u=u_{\mathrm{I}}$ satisfies the equation
    \begin{equation}\label{discrete_biased_infty}
        u(x)=\left(\frac{\rho}{\rho+1} \sup _{y:(x, y) \in E_1} u(y)+ \frac{1}{\rho+1}\inf _{y:(x, y) \in E_2} u(y)\right)+f(x)
    \end{equation}

for every non-terminal state $x \in X \backslash Y$ for which the right-hand-side is well-defined, and $u_{\mathrm{I}}(x)=-\infty$ when the right-hand-side is of the form $(\infty+(-\infty))+f(x)$. The analogous statement holds for $u_{\mathrm{II}}$, except that $u_{\mathrm{II}}(x)=+\infty$ when the right-handside is of the form $(\infty+(-\infty))+f(x)$.

\end{lemma}

In this work we shall focus on the case $f=0$ which has nice variational structure.
%

\subsection{Biased tug-of-war on length space}

Given a length space $X$, some $Y \subset X$, a real function $g$ on $Y$, a small $\epsilon > 0$, and any $\rho(\epsilon) > 0$, the \textit{biased $\epsilon$-tug-of-war game} is defined as follows. The starting position is $x_0 \in X \setminus Y$. At the $k$th step the two players toss a biased coin which player I wins with odds of $\rho(\epsilon)$ to $1$, i.e., with probability $\rho(\epsilon)/(\rho(\epsilon) + 1)$, and the winner chooses $x_k$ with $d(x_k, x_{k-1}) < \epsilon$. (Even if $X \subset \mathbb{R}^n$, the moves are measured in the path metric of $X$, not in the Euclidean metric.) The game ends when $x_k \in Y$, and player II pays the amount $g(x_k)$ to player I. Assuming that the value of this game exists (defined and proved below), it is denoted by $u^{\epsilon}(x_0)$.

\medskip

Note that if the probability for player I to win a coin toss is $\dfrac{1 + \theta(\epsilon)}{2}$, then we have
\[
\rho(\epsilon) = \frac{1 + \theta(\epsilon)}{1 - \theta(\epsilon)} 
\quad \text{and} \quad
\theta(\epsilon) = \frac{\rho(\epsilon) - 1}{\rho(\epsilon) + 1}.
\]
We also denote the reverse odd-ratio
$$
r(\epsilon) := \frac{1}{\rho(\epsilon)}.
$$

Throughout this paper we shall assume
\[
\rho(\epsilon) := e^{\beta \epsilon}
\quad \text{and} \quad
\theta(\epsilon) := \tanh\!\left(\frac{\beta \epsilon}{2}\right)
= \frac{e^{\beta \epsilon} - 1}{e^{\beta \epsilon} + 1}
= \frac{\beta \epsilon}{2} - \frac{(\beta \epsilon)^3}{24} + O((\beta \epsilon)^5).
\]

The importance of such choice of $\rho$ is already mentioned in \cite{peres2010biased}. We will see this choice of $\rho$ is consistent with our $\beta$-AM.

The value of the game $u_I^{\epsilon}, u_{II}^{\epsilon}$ are defined similarly as before. From now on in this subsection we assume $f=0$.

\begin{theorem}\label{main_tug_of_war}\cite{peres2010biased}
    Under the same assumption as in \cref{main}, there exists an unique $\beta$-AM of $g$, as the limit of the $u^{\epsilon}$. 
\end{theorem}

\begin{theorem}\label{u_I_theorem}
    In the general length space $X$ and $g$ is bounded below, there exists an $\beta$-AM of $g$, as the limit of the $u^{\epsilon}_I$. In Euclidean space, this corresponds to the smallest viscosity solution to the biased infinity Laplace equation.
\end{theorem}
By considering the natural coupling of restricted tug-of-war with different bias, we have the following corollary: 
\begin{corollary}
    For $\beta_1$-extension $u_1$ and $\beta_2$-extension $u_2$ with the same boundary data $g$ and $ 0 \leq \beta_1 \leq \beta_2$, we have $u_1 \leq u_2$.
\end{corollary}

The reader might wonder why this corresponds to Aronsson equation with $H(\nabla u, u) = |\nabla u| + \beta u$ and we give some heuristics. Our discovery is motivated by the biased tug of war. Note that we can rewrite \cref{discrete_biased_infty} as ($f=0$)
$$
u(x) - r\min_{y\sim x} u(y) = \max_{y\sim x} u(y)  - ru(x) 
$$
Thus the quantity $\delta(x) = u(x) - r\min_{y\sim x} u(y)$ under jointly optimal play is non-decreasing along the trajectory. Note that 
$\frac{\delta(x)}{\epsilon} = \frac{u(x) - e^{-\beta \epsilon}\min_{y\sim x} u(y)}{\epsilon} \to |\nabla u|(x) + \beta u(x)$ as $\epsilon \to 0$, if we assume $u$ is sufficient smooth.

\subsection{Regularity}

One can utilize the absolute minimizing property to study the regularity of biased infinity harmonic functions. As far as we are aware, these results are new as the Exponential Absolute Minimizing property was not realized before. The following result shows that any blow up limit is linear, which is the same as in the unbiased case. 

\begin{proposition}\label{linear-blow-up}
    Let $u $ be absolute minimal $\beta$-extension in $U$, for $U \subset \mathbb{R}^n$, and $x^0 \in U$. Let $|\cdot|$ be the Euclidean norm. If $\lambda_j>0$ satisfies $\lambda_j \downarrow 0, v \in C\left(\mathbb{R}^n\right)$, and

$$
v(x)=\lim _{j \rightarrow \infty} \frac{u\left(\lambda_j x+x^0\right)-u\left(x^0\right)}{\lambda_j}
$$

holds uniformly on compact subsets of $\mathbb{R}^n$, then $v$ is a linear function.
\end{proposition}

\subsection{Background and literature}
Analogous to the the connection of Brownian motion and harmonic functions, Peres, Schramm, Sheffield, and Wilson \cite{PSSW} revealed a connection between infinity harmonic function and stochastic games called Tug-of-war. Since then, there has been intensive studies of connection of stochastic games and and nonlinear PDEs, e.g. \cite{peres2008tug} for $p$-harmonic function and \cite{bjorland2012nonlocal} for fractional infinity Laplacian. See \cite{blanc2019game} for a compresentive introduction. The biased infinity harmonic function is introduced by \cite{peres2010biased} via a probabilistic argument by allowing bias in the coins of tug-of-war, and has been further studied in \cite{armstrong2011infinity_gradient, liu2021weighted_ev, liu2018inhomogeneous, liu2019parabolic, peres2019biased_graph, liu2022regularity}. See \cite{fu_Hammond_Pete_2022stake} for a recent interpretation of biased infinity harmonic functions as value functions of stake-governed Tug-of-war, where each player has a given budget to play the game, and at each turn the bias of the coin is determined by how much each player stakes.

The study of infinity harmonic functions and absolute minimizing Lipschitz extension was initiated by Aronsson \cite{aronsson1967extension,aronsson1968partial,aronsson1969minimization,aronsson1984certain}, as the prototype of calculus of variation in $L^{\infty}$ norm. There have been great
 interests in the last few years to study the minimization problem of the supremal functional, for a Hamiltonian $H(x,z,p)$,
:
 $$
 F(u, \Omega)=\underset{x \in \Omega}{\operatorname{ess} \sup } H(x, u(x), \nabla u(x)), \quad \Omega \subset \mathbf{R}^n, u \in W^{1, \infty}\left(\Omega \right) .
$$
It was shown in
\cite{barron2008infinity_generalization,crandall2003efficient} when $H \in C^1$  that any absolute minimizer for $H$ is a viscosity solution of the Aronsson equation:
$$
H_p(D u(x), u(x), x) \cdot \nabla_x(H(D u(x), u(x), x)) =0 \quad \text { in } \Omega.
$$
One of the main open problems in the area is regularity. In particular, it is unknown if the infinity harmonic function is $C^1$ in dimension greater than $2$. See \cite{evans2008c1_alpha, savin2005c} for $C^1$ regularity in dimension $2$ and \cite{evans2011everywhere} for everywhere differentiability in any dimension. In the case $H(x, u(x), \nabla u(x)) = H(\nabla u(x))$ or $H(x, u(x), \nabla u(x)) = H(x, \nabla u(x))$, i.e. $H$ has no dependence on $u$, the above regularities results were generalized to general Hamiltonian $H(\cdot)$ with suitable convexity or quasi-convexity assumptions. See \cite{champion2007principles, c_1_arrosson_wang2008c,everywhere_arrosson_peng2021}. 

Another main issue is uniqueness of the Aronsson equation. When $$H(x, u(x), \nabla u(x)) = H(\nabla u(x)) \text{ or } H(x, \nabla u(x)),$$ the uniqueness is established in great generality in recent years as in \cite{armstrong2011convexity, uniqueness_nonunique_Arrosson_jensen2008, koskela2014intrinsic, x_dependent_uniqueness_miao2017}

However, up to our knowledge there is not much progress on the above two issues when $H(\cdot )$ has $u$ dependence. Neither there has been a variational characterization of biased infinity harmonic functions, except in \cite{peres2019biased_graph} where they introduced a notion of $r$-slope of biased infinity harmonic functions on finite graphs, which could be viewed as a discrete analogy of the $\beta$-slope we introduced in this paper. As shown in Theorem \cref{main}, the $\beta$-AM is the same as the biased infinity harmonic function in Euclidean space with Euclidean norm, which is the Aronsson equation for $H(u, \nabla u) = |\nabla u| + \beta u.$ Note that $H$ is not $C^1$ when $\nabla u=0$, thus previous results \cite{crandall2003efficient, crandall2009derivation} do not apply here. As we will see later the biased AM has many properties similar to the unbiased one, and we believe this and other relative notion, such as biased McShane-Whitney extension and bias convexity, can pave a way to further study on Aronsson equations with $u$-dependent Hamiltonian. In particular, we expect the framework in this paper could cover all Hamiltonians of the type $H(x, u(x), \nabla u(x)) = J(x, \nabla u(x)) + \beta u(x)$ with suitable quasiconvex and coercive conditions on $J$. We shall explore this direction in a future work.

\section{Restricted biased Tug-of-war and viscosity solutions }

We recall some definition about viscosity solutions.

\begin{definition}\label{def_viscosity_sol}
Let \( U \subseteq \mathbb{R}^n \) be a domain. 
An upper semicontinuous function \( u \) is a viscosity solution of 
\( \Delta_{\infty}^{\beta} u \ge 0 \) in \( U \), or, in other words, 
a viscosity subsolution of $ \Delta_{\infty}^{\beta} u = 0 $, if for every local maximum point 
\( \tilde{x} \in U \) of \( u - w \), where \( w \) is \( C^2 \) in some neighbourhood of \( \tilde{x} \), 
we have
\[
\begin{cases}
\Delta_\infty^{\beta} w(\tilde{x}) \ge 0, & \text{if } \nabla w(\tilde{x}) \ne 0, \\[6pt]
D^2 w(\tilde{x}) \ge 0, & \text{if } \nabla w(\tilde{x}) = 0.
\end{cases}
\]

Similarly, a lower semicontinuous function \( u \) is a viscosity solution of 
\( \Delta_{\infty}^{\beta} u \le 0 \), or viscosity supersolution of $ \Delta_{\infty}^{\beta} u = 0 $, 
if for every local minimum point \( \tilde{x} \in U \) of \( u - w \) we have 
\[
\begin{cases}
\Delta_\infty^{\beta} w(\tilde{x}) \le 0, & \text{if } \nabla w(\tilde{x}) \ne 0, \\[6pt]
D^2 w(\tilde{x}) \le 0, & \text{if } \nabla w(\tilde{x}) = 0.
\end{cases}
\]
A function is a viscosity solution of \( \Delta_{\infty}^{\beta} u = 0 \) 
if it is continuous and both a viscosity supersolution and subsolution.
\end{definition}

\begin{definition}\label{def_mean_value}
A function $u \in C(\Omega)$ satisfies the formula
\[
u(x) = \frac{ \displaystyle (1+ \beta \epsilon/2) \max_{B(x,\varepsilon)} \{u\} + (1-\beta \epsilon/2) \min_{B(x,\varepsilon)} \{u\}}{2} + o(\varepsilon^2)
\quad \text{as } \varepsilon \to 0
\]
in the \emph{viscosity sense}, if the two conditions below hold:
\begin{itemize}
    \item If $x_0 \in \Omega$ and if $\phi \in C^2(\Omega)$ touches $u$ from below at $x_0$, then
    \[
    \phi(x_0) \ge 
    \frac{1}{2} \left( (1+ \beta \epsilon/2)
    \max_{|y - x_0| \le \varepsilon} \{\phi(y)\}
    + (1- \beta \epsilon/2)
    \min_{|y - x_0| \le \varepsilon} \{\phi(y)\}
    \right)
    + o(\varepsilon^2)
    \quad \text{as } \varepsilon \to 0.
    \] 
    Moreover, if it so happens that $\nabla \phi(x_0) = 0$, we require that the test function satisfies
    \[
    D^2 \phi(x_0) \le 0.
    \]

    \item If $x_0 \in \Omega$ and if $\psi \in C^2(\Omega)$ touches $u$ from above at $x_0$, then
    \[
    \psi(x_0) \le 
    \frac{1}{2} \left( (1+ \beta \epsilon/2)
    \max_{|y - x_0| \le \varepsilon} \{\psi(y)\}
    + (1- \beta \epsilon/2)
    \min_{|y - x_0| \le \varepsilon} \{\psi(y)\}
    \right)
    + o(\varepsilon^2)
    \quad \text{as } \varepsilon \to 0.
    \]
    Moreover, if it so happens that $\nabla \psi(x_0) = 0$, we require that the test function satisfies
    \[
    D^2 \psi(x_0) \ge 0.
    \]
\end{itemize}
\end{definition}
From the definition the proof of \cref{thm_biased_infty_lap} $\Leftrightarrow$ \cref{thm_mean_value} is straightforward in the following.

\begin{proof}[Proof of \cref{thm_biased_infty_lap} $\Leftrightarrow$ \cref{thm_mean_value}]
Note that for a $C^2$ function $g$ and $\nabla g (x_0) \not =0$, we have 
as $\ \varepsilon \rightarrow 0$,
$$
g\left(x_0\right)=\frac{ {\max}_{B\left(x_0, \varepsilon\right)}\{g\}+ {\min}_{B\left(x_0, \varepsilon\right)}\{g\}}{2} - \frac{\Delta^N_{\infty} g\left(x_0\right)}{2} \varepsilon^2+o\left(\varepsilon^2\right),$$

$${\max}_{B\left(x_0, \varepsilon\right)}\{g\} = g(x + \epsilon \frac{\nabla g}{|\nabla g|}(x) + o(\epsilon)) = g(x)+\epsilon |\nabla g|(x) + o(\epsilon),$$
$${\min}_{B\left(x_0, \varepsilon\right)}\{g\} = g(x - \epsilon \frac{\nabla g}{|\nabla g|}(x) + o(\epsilon)) = g(x) - \epsilon |\nabla g|(x) + o(\epsilon).$$
Thus 
$$
\epsilon \left( {\max}_{B\left(x_0, \varepsilon\right)}\{g\} - {\min}_{B\left(x_0, \varepsilon\right)}\{g\} \right) 
= 2 \epsilon^2 | \nabla g(x)| + o(\epsilon^2)
$$
Combine together we get
\begin{equation}\label{C^2_biased_mean}
\begin{aligned}
 g\left(x_0\right) &=
\frac{1}{2} \left[ {(1+\beta \epsilon/2) {\max}_{B\left(x_0, \varepsilon\right)}\{g\}+ (1-\beta \epsilon/2){\min}_{B\left(x_0, \varepsilon\right)}\{g\}} \right] \\
&- ({\Delta^N_{\infty} g\left(x_0\right) + \beta |\nabla g(x_0)|}) \varepsilon^2 + o\left(\varepsilon^2\right)
\end{aligned}
\end{equation}

Let $u$ satisfy the biased mean value formula in the viscosity sense:
\[
u(x) = \frac{(1+\tfrac{\beta\varepsilon}{2})\max_{B(x,\varepsilon)}u
+ (1-\tfrac{\beta\varepsilon}{2})\min_{B(x,\varepsilon)}u}{2}
+ o(\varepsilon^2)
\quad \text{as } \varepsilon \to 0.
\]
Let $\psi \in C^2(\Omega)$ touch $u$ from above at $x_0 \in \Omega$, and set
\[
M_\varepsilon := \max_{B(x_0,\varepsilon)} \psi, 
\qquad 
m_\varepsilon := \min_{B(x_0,\varepsilon)} \psi.
\]
Then, by the viscosity interpretation,
\[
\psi(x_0) \le 
\frac{(1+\tfrac{\beta\varepsilon}{2})M_\varepsilon
+ (1-\tfrac{\beta\varepsilon}{2})m_\varepsilon}{2}
+ o(\varepsilon^2).
\]
Rearranging and dividing by $\varepsilon^2$ gives
\begin{equation}\label{eq:main-bias}
0 \le 
\liminf_{\varepsilon \to 0}
\left\{
\frac{1}{\varepsilon^2}
\Big(\tfrac{M_\varepsilon + m_\varepsilon}{2} - \psi(x_0)\Big)
+ \frac{\beta}{4\varepsilon}(M_\varepsilon - m_\varepsilon)
+ o(1)
\right\}.
\end{equation}

\medskip\noindent
\textit{Case 1:} $\nabla \psi(x_0) \neq 0$.
This comes from the above \cref{C^2_biased_mean}.

\medskip\noindent
\textit{Case 2:} $\nabla \psi(x_0) = 0$. By definition $D^2 \psi \geq 0$.

Therefore $u$ is a subsolution.

The supersolution case (test function touching from below) is analogous and yields the reverse inequality.
\end{proof}

\section{Basic property of $\beta$-slope and $\beta$-extension}\label{sec_properties}

\begin{lemma}\label{Family}
    For a family of functions $\mathcal{F}$ such that $\forall u \in \mathcal{F}$, $L^{\beta}_f = L^\beta$, i.e. the $\beta$-exponential constant is the same, then $u_{sup}(x) := \sup_{u\in \mathcal{F}} u(x) $ and $u_{inf} (x) :=\inf_{u\in \mathcal{F}} u (x) $ has $\beta$-slope as $L^{\beta}$ as well.
\end{lemma}

\begin{proof}
    Note that $\forall \epsilon>0, \forall x$ there is there is $u_i$ such that $u_{sup}(x) \leq  u_i(x) + \epsilon $, 
    $$
    \frac{u_{sup}(x)-e^{-\beta|x-y|} u_{sup}(y)}{1-e^{-\beta|x-y|}} \beta \leq 
    \frac{u_{i}(x) +\epsilon -e^{-\beta|x-y|} u_i(y)}{1-e^{-\beta|x-y|}} \beta
    \leq L^{\beta} + \frac{\epsilon}{1-e^{-\beta|x-y|}} \beta.
    $$
    As $\epsilon \to 0$, the desired result follows.

    Similarly, $\forall \epsilon>0, \forall x$ there is there is $u_i$ such that $u_{inf}(x) \geq  u_i(x) - \epsilon $, 
    $$
    \frac{u_{inf}(x)-e^{-\beta|x-y|} u_{inf}(y)}{1-e^{-\beta|x-y|}} \beta \leq 
    \frac{u_{i}(x)  -e^{-\beta|x-y|} u_i(y) + e^{-\beta|x-y|} \epsilon}{1-e^{-\beta|x-y|}} \beta
    \leq L^{\beta} + \frac{e^{-\beta|x-y|} \epsilon}{1-e^{-\beta|x-y|}} \beta.
    $$
\end{proof}

\begin{proposition}
    Recall \cref{def_Mchshane_Whit}:
    $$\begin{aligned} 
    & \Psi(g)(x):=\inf _{y \in Y}\left(  e^{-\beta |x-y|}g(y)
    + \frac{L_g(Y)}{\beta} (1-e^{-\beta |x-y|}))\right), \\ 
    & \Lambda(g)(x):=\sup _{y \in Y}
    \left( e^{\beta |x-y|}g(y)
    + \frac{L_g(Y)}{\beta} (1-e^{\beta |x-y|}))
    \right).
    \end{aligned}
    $$

    If $u$ is any extension of $g$ such that $L_u (X) = L_g(Y)$, then $$\Lambda(f) \leq u \leq \Psi(f).$$ For this reason $\Psi(f)$ is called the maximal extension of $f$ and $\Lambda(f)$ is the minimal extension.
\end{proposition}

\begin{proof}
    It's clear that given any $\beta$-extension $u$, $\forall y, z \in Y$ we have 
    $$
    \left( e^{\beta |x-y|}g(y)
    + \frac{L_g(Y)}{\beta} (1-e^{\beta |x-y|}))
    \right) \leq 
    u(x) \leq \left(  e^{-\beta |x-z|}g(y)
    + \frac{L_g(Y)}{\beta} (1-e^{-\beta |x-z|}))\right)
    $$
    The desired result follows by taking sup on the LHD and inf on the RHS, and \cref{Family}.
\end{proof}

\begin{proposition}\label{equality_cone_on_boundary}
    Cone functions also have the following property: if $W \subset \subset \Omega, C$ is a cone function with vertex $z \notin W$ and slope $a$, and $u \in C(\bar{W})$ satisfies $u=C$ on $\partial W$ and $L_u(W)=|a|$, then $u \equiv C$ on $W$. 
\end{proposition}

\begin{proof}
    We may assume $\beta >0$ without loss of generality. Note that the case for $\beta <0$ is the same by switching positive/negative cones. 

    First, we assume $\Omega = \mathbb{R}^n$. 
    
    To establish this, we show that $y \in W$ and $u(y) \neq C(y)$ is impossible. Without loss of generality we assume $a \geq 0$, as for $a<0$ the proof is essentially flip the sign of inequality.
    
    To rule out $u(y)>C(y)$, let $w$ be the vertex of the exponential cone $C$, $y^*, y^{* *}$ be the two points on the boundary such that $w, y, y^*, y^{* *}$ are in the same line. Then

$$
u(y)- e^{-\beta |y-y^*|} u\left(y^*\right)=u(y)- e^{-\beta |y-y^*|} C\left(y^*\right)>C(y)- e^{-\beta |y-y^*|} C\left(y^*\right)
$$
$$
=C_{z}^{+}(y)+B e^{-\beta |y-z|} - e^{-\beta |y-y^*|}\left( C_{z}^{+}(y^*) - B e^{-\beta |y^{*}-w|} \right)  =a (1-e^{-\beta |y-y^*|}).
$$

which shows that $L_u(W)>a$, a contradiction. To rule out $u(y)<C(y)$, argue similarly 
$$
u\left(y^{**}\right)- e^{-\beta |y-y^{**}|} u (y) =C(y^{**})- e^{-\beta |y-y^{**}|} u(y)
$$
$$
>C(y^{**})- e^{-\beta |y-y^{**}|} C\left(y \right)=a (1-e^{-\beta |y-y^{**}|}).
$$
This finishes the proof when $\Omega = \mathbb{R}^n$.

For the general case in length space, note that, given $y \in W$, if we can find $y^*, y^{* *} \in \partial W$ such that $d(y,z) = d(y,y^{*})+ d(y^{* },z)$ and $d(y,z) = - d(y,y^{**}) + d(y^{** },z)$, the above proof works through in the same fashion.

Given any $x\in \partial W$, by triangle inequality we have 
$$
d(z,y) \leq d(z,x) + d(x,y)
$$
$$
d(z,x) - d(y,x)\leq d(z,y) 
$$
We claim that $$y^* = \arg \min_{x\in \partial W} d(z,x) + d(x,y)$$
and $$y^{**} = \arg \max_{x\in \partial W} d(z,x) - d(y,x)$$
will be the desired choice (note the extreme values are achieved due to compactness). I.e. we claim both inequalities are achieved as equalities, respectively by $y*$ and $y^{**}$. If not, then for some $\delta>0$, $d(z,y^{*}) + d(y,y^{*}) - \delta =d(z,y) \leq d(z,y^{*}) + d(y,y^{*})
$, which leads to a contradiction. Note that $d(z,y^{*}) = \inf \{|\gamma|: \gamma: \text{continuous path from $z$ to $y^{*}$} \}$, we also have $d(z,y^{*}) \geq d(z,y)+ d(y,y^{*})$
Similarly, if
$$
d(z,y^{**}) = d(z,y)+ d(y,y^{**}) - \delta 
$$
However, as $d(z,y^{**}) = \inf \{|\gamma|: \gamma: \text{continuous path from $z$ to $y^{**}$} \}$, we also have $d(z,y^{**}) \geq d(z,y)+ d(y,y^{**})$, which leads to a contradiction.
\end{proof}

\subsection{Slope estimate}

\begin{definition}
    We define the positive $\beta$-slope with radius $r$ at $y$ as 
$$
S^{+}_{\beta,u}(y, r):=\max _{\{w:|w-y| \leq r\}} \left( \frac{u(w)- e^{-\beta r}u(y)}{1- e^{-\beta r}} \beta \right).
$$
It's easy to see it's nondecreasing in $r$ for $0<r<\operatorname{dist}(y, Y)$. And $S^{+}_{\beta,u}(y, r) \geq \beta u(y).$ We further define the positive $\beta$-slope at $y$ as the limit
\begin{equation}\label{local_slope}
    S^{+}_{\beta,u}(y):=\lim _{r \downarrow 0} S^{+}_{\beta,u}(y, r)=\inf _{0<r<\operatorname{dist}(y, Y)} S^{+}_{\beta,u}(y, r)
\end{equation}

is therefore well-defined and finite. 

Similarly we define the negative $\beta$-slope with radius $r$ at $y$ as :
\begin{align}\label{negative_slope_by_radius}
    S^{-}_{\beta,u}(y, r)
    &:=\min _{\{w:|w-y| \leq r\}}\left( \frac{e^{-\beta r} u(w)- u(y)}{1- e^{-\beta r}} \beta \right) \\
    &= - \max _{\{w:|w-y| \leq r\}}\left( \frac{u(y) - e^{-\beta r} u(w)}{1- e^{-\beta r}} \beta \right)
\end{align}

And the negative $\beta$-slope at $y$ as the limit
\begin{equation}\label{local_slope_negative}
    S^{-}_{\beta,u}(y):=\lim _{r \downarrow 0} S^{-}_{\beta,u}(y, r)=\inf _{0<r<\operatorname{dist}(y, Y)} S^{-}_{\beta,u}(y, r)
\end{equation}

Again, We shall suppress the bias $\beta$ and function $u$ and write $S^+$ or $S^-$, when there is no confusion.

\end{definition}

\begin{lemma}\label{CECA_convex}
    Let $u \in \mathrm{CECA}(U)$ and $y \in U$. Then
    \begin{equation}\label{maximum_on_sphere}
        \max _{\{w:|w-y|=r\}} u(w)=\max _{\{w:|w-y| \leq r\}} u(w) \quad for \quad 0 \leq r<\operatorname{dist}(y, \partial U)
    \end{equation}

Moreover,
\begin{equation}\label{bound_by_slope}
u(x) \leq  e^{-\beta |x-y|}u(y)+ \left({1- e^{-\beta |x-y|}}\right)  \max _{\{w:|w-y|=r\}} \left( 
\frac{u(w)- e^{-\beta r}u(y)}{1- e^{-\beta r}}
\right)  
\end{equation}
$\text{ for }  |x-y| \leq r<\operatorname{dist}(y, \partial U)$.

And the function
$$
g^{+}(r):=\max _{\{w:|w-y| \leq r\}} u(w)=\max _{\{w:|w-y|=r\}} u(w)
$$
is $\beta$-convex: for $0 \leq s \leq t \leq  r<\operatorname{dist}(y, Y):$ 

\begin{equation}\label{slope_convex_1}
    g^{+}(t) \leq  \frac{e^{\beta(r-t)}-1}{e^{\beta (r -s)}-1} g^{+}(s) + \frac{e^{\beta(r-s)} - e^{\beta(r-t)}}{e^{\beta (r -s)}-1} g^{+}(r)
\end{equation}
Equivalently, 
\begin{equation}\label{slope_convex_2}
    h^{+}(t):=e^{\beta t}g^{+}(t) 
    \leq  \frac{e^{\beta(r-s)} - e^{\beta(t-s)}}{e^{\beta (r -s)}-1} e^{\beta s} g^{+}(s) + \frac{e^{\beta(t-s)} - 1}{e^{\beta (r -s)}-1} e^{\beta r} g^{+}(r)
\end{equation}
Consequently, the function $I^{+}(t) := h^{+}( \frac{\log(1+t)}{\beta} )$ is convex on $0 \leq t<\operatorname{dist}(y, Y)$.

Finally,

\begin{equation}\label{slope_by_radius}
    S^{+}_{\beta,u}(y, r):=\max _{\{w:|w-y|=r\}}\left( \frac{u(w)- e^{-\beta r}u(y)}{1- e^{-\beta r}} \beta \right)
\end{equation}
is nondecreasing in $r, 0<r<\operatorname{dist}(y, Y)$. Also,
\begin{equation}\label{slope_lower_bound}
    S^{+}_{\beta,u}(y, r) \geq \beta u(y).
\end{equation}
\end{lemma}

\begin{proof}
    Note that \cref{maximum_on_sphere} follows easily by the constant cone function defined as the maximum on the sphere. \cref{bound_by_slope} also immediately follows by \cref{maximum_on_sphere} and CECA.

    Also, note that for any constant $M$, if we replace $u$ by $u+M$, then $g^{+}$ is replaced by $g^{+}+M$. Thus \cref{slope_convex_1} is invariant under adding a constant to $u$. Therefore we may assume $u \geq 0$ without loss of generality. 
    
    We claim that for $0 \leq s\leq|x-y| \leq r<\operatorname{dist}(y, Y)$,
$$
u(x) \leq e^{-\beta (|x-y| -s)}g^{+}(s) + \frac{g^{+}(r)- e^{-\beta (r -s)}g^{+}(s)}{1- e^{-\beta (r -s)}} \left({1- e^{-\beta (|x-y|-s)}}\right)
$$
By assuming $u \geq 0$, $A = \frac{g^{+}(r)- e^{-\beta (r -s)}g^{+}(s)}{1- e^{-\beta (r -s)}} \geq 0$.

Indeed, by the definition of $g^{+}$, the cone function on the right bounds $u$ above on the boundary of the annular region $s<|x-y|<r$.
If $|x-y|= t$, we have RHS equals
$$
\frac{e^{-\beta (t -s)}- e^{-\beta (r -s)}}{1- e^{-\beta (r -s)}} g^{+}(s)
+
\frac{1- e^{-\beta (t -s)}}{1- e^{-\beta (r -s)}} g^{+}(r)
$$
Optimizing over $|x-y|= t$ gives
\begin{align}
g^{+}(t) &\leq  \frac{e^{-\beta (t -s)}- e^{-\beta (r -s)}}{1- e^{-\beta (r -s)}} g^{+}(s)
+
\frac{1- e^{-\beta (t -s)}}{1- e^{-\beta (r -s)}} g^{+}(r) \\
&=  \frac{e^{\beta(r-t)}-1}{e^{\beta (r -s)}-1} g^{+}(s) + \frac{e^{\beta(r-s)} - e^{\beta(r-t)}}{e^{\beta (r -s)}-1} g^{+}(r)
\end{align}

Note that by \cref{bound_by_slope}, the monotonicity follows easier by optimizing over $|x-y|=s\leq r$. 
Another proof of the monotonicity is given by the $\beta$-convexity: by setting $s=0$ in \cref{slope_convex_2} we have, for $ 0<t \leq r$
$$ \frac{1}{\beta} S^{+}_{\beta,u}(y,t) 
= \frac{g^{+}(t) - e^{-\beta t}g^{+}(0)}{1-e^{-\beta t}} 
\leq \frac{g^{+}(r) - e^{-\beta r}g^{+}(0)}{1-e^{-\beta r}} 
= \frac{1}{\beta} S^{+}_{\beta,u}(y,r) $$

Note that \cref{slope_lower_bound} is simply given by letting $w \to y$ as $r \downarrow 0$ and the above monotonicity.

\end{proof}

We have the following lemma analogues to above: 

\begin{lemma}
    Let $u \in \mathrm{CECB}(U)$ and $y \in U$. Then
    \begin{equation}\label{min_on_sphere}
        \min_{\{w:|w-y|=r\}} u(w)=\min _{\{w:|w-y| \leq r\}} u(w) \quad for \quad 0 \leq r<\operatorname{dist}(y, Y)
    \end{equation}

Moreover,
\begin{equation}\label{bound_by_slope_negative}
u(x) \geq  e^{\beta |x-y|}u(y)
+ \left({ 1- e^{\beta |x-y|}}\right)
\max _{\{w:|w-y|=r\}}\left( \frac{u(w) - e^{\beta r} u(y)}{1- e^{\beta r}}  \right)
\end{equation}
$\text{ for }  |x-y| \leq r<\operatorname{dist}(y, Y)$.

And the function
$$
g^{-}(r):=\min _{\{w:|w-y| \leq r\}} u(w)=\min _{\{w:|w-y|=r\}} u(w)
$$
is $\beta$-concave: for $0 \leq s \leq t \leq  r<\operatorname{dist}(y, Y):$ 

\begin{equation}\label{slope_convex_1}
    g^{-}(t) 
    \geq  \frac{e^{-\beta(r-t)}-1}{e^{-\beta (r -s)}-1} g^{+}(s) + \frac{e^{-\beta(r-s)} - e^{-\beta(r-t)}}{e^{-\beta (r -s)}-1} g^{+}(r)
\end{equation}
Equivalently, 
\begin{equation}\label{slope_convex_2}
    h^{-}(t) := e^{- \beta t}g^{-}(t) \geq  
    \frac{e^{-\beta(r-s)} - e^{-\beta(t-s)}}{e^{-\beta (r -s)}-1} e^{-\beta s} g^{+}(s) + \frac{e^{-\beta(t-s)} - 1}{e^{-\beta (r -s)}-1} e^{-\beta r} g^{+}(r)
\end{equation}

Consequently, the function $I^{-}(t) := h^{-}( \frac{\log(1+t)}{\beta} )$ is concave on $0 \leq t<\operatorname{dist}(y, Y)$.

Finally,

\begin{equation}
    S^{-}_{\beta,u}(y, r):= \max _{\{w:|w-y|=r\}}\left( \frac{u(w) - e^{\beta r} u(y)}{1- e^{\beta r}} \beta  \right)
\end{equation}
is nonincreasing in $r, 0<r<\operatorname{dist}(y, Y)$, and 
\begin{equation}\label{negative_slope_upperbound}
    S^{-}_{\beta,u}(y, r) \leq \beta u(y).
\end{equation}

\end{lemma}

\begin{proof}
    Note that \cref{min_on_sphere}, \cref{bound_by_slope_negative} follows similarly as in the previous lemma.

    We claim that for $0 \leq s\leq|x-y| \leq r<\operatorname{dist}(y, Y)$,

$$
u(x) \geq e^{\beta (|x-y| -s)}g^{-}(s)
+ \frac{g^{-}(r) - e^{\beta (r -s)} g^{-}(s) }{1- e^{\beta (r -s)}} \left({ 1- e^{\beta (|x-y|-s)}} \right)
$$
Similarly as before, we may assume $u \geq 0$ without loss of generality, thus
$A = \frac{g^{-}(r) - e^{\beta (r -s)} g^{-}(s) }{1- e^{\beta (r -s)}} \geq 0.$

Indeed, by the definition of $g^{-}$, the cone function on the right bounds $u$ below on the boundary of the annular region $s<|x-y|<r$.
If $|x-y|= t$, we have RHS equals
$$
\frac{e^{\beta (t -s)} -e^{\beta (r -s)} }{1-e^{\beta (r -s)}} g^{-}(s)
+
\frac{1-e^{\beta (t -s)} }{1-e^{\beta (r -s)}} g^{-}(r)
$$
Optimizing over $|x-y|=t$, we see $g^-$ is $\beta$-concave
\begin{align*}
    g^{-}(t) &\geq \frac{e^{\beta (t -s)} -e^{\beta (r -s)} }{1-e^{\beta (r -s)}} g^{-}(s)
+
\frac{1-e^{\beta (t -s)} }{1-e^{\beta (r -s)}} g^{-}(r) \\
& = \frac{e^{-\beta(r-t)}-1}{e^{-\beta (r -s)}-1} g^{+}(s) + \frac{e^{-\beta(r-s)} - e^{-\beta(r-t)}}{e^{-\beta (r -s)}-1} g^{+}(r). 
\end{align*}

Again by \cref{min_on_sphere} \cref{bound_by_slope_negative},  the monotonicity follows easily by optimizing over $|x-y|=s\leq r$. Another proof of the monotonicity is given by the $\beta$-concavity, in the same fashion as in the previous lemma.

\cref{negative_slope_upperbound} is simply given by letting $w \to y$ as $r \downarrow 0$ and the above monotonicity.
\end{proof}

The following is a Harnack-type inequality, which establishes the local uniform continuity of $u$:

\begin{lemma}\label{Harnack_lemma}
    Let $u \in C(U)$ satisfy \cref{bound_by_slope}, then for $z \in U$, $x,y \in B_R (z)$ and  $R < d(z, Y)/4$, we have

\begin{equation}\label{Harnack_1}
    \sup_{B_R(z)} u  - \sup_{B_{4R}(z)} u
     \leq  \frac{e^{R \beta} - 1}{e^{3R \beta} -1}  (\inf_{B_R(z)} u -\sup_{B_{4R}(z)} u)
\end{equation}

and 
\begin{equation}\label{slope_Harnack}
    \frac{u(x) -  u(y)e^{-\beta |x-y|}}{1- e^{-\beta |x-y|}} \beta
\leq\beta\sup_{B_{4R}(z)} u+ \left( \sup_{B_{4R}(z)} u - \sup_{B_R(z)} u \right) \frac{\beta}{e^{R \beta}-1}
\end{equation}

If $u$ is non-positive, the above two reduce to:

\begin{equation}\label{Harnack_nonpositive}
    \sup_{B_R(z)} u  
     \leq  \frac{e^{R \beta} - 1}{e^{3R \beta} -1} \inf_{B_R(z)} u 
\end{equation}

\begin{equation}\label{slope_Harnack_nonpositive}
    \frac{u(x) -  u(y)e^{-\beta |x-y|}}{1- e^{-\beta |x-y|}} \beta
\leq  \left( - \sup_{B_R(z)} u \right) \frac{\beta}{e^{R \beta}-1}
\end{equation}
\end{lemma}

\begin{proof}
We first assume $u$ is non-positive.
    $$
u(x) \leq  u(y) \left(e^{-\beta |x-y|} -  
\frac{ e^{-\beta r}}{1- e^{-\beta r}}  \left({1- e^{-\beta |x-y|}}\right)
\right) 
    $$
    $$
u(x) -  u(y)e^{-\beta |x-y|} \leq  - u(y) \left(  
\frac{ e^{-\beta r}}{1- e^{-\beta r}}  \left({1- e^{-\beta |x-y|}}\right)
\right) 
    $$
    Let $r \uparrow d(y)$. Thus $d(y, Y) \geq 3 R, |x-y| \leq 2 R$, we have 
    $$
u(x) - e^{-2R \beta} u(y) \leq u(x) -  u(y)e^{-\beta |x-y|} \leq  - u(y) \left(  
\frac{1}{e^{3R \beta } - 1}  \left({1- e^{-2R \beta}}\right)
\right) 
    $$
    Thus we have 
    $$
    u(x) \leq u(y) e^{-2R \beta}\left( 1- \frac{e^{2R \beta} - 1}{e^{3R \beta} -1} \right)
    $$
    and 
    \begin{align}
    \sup_{B_R(z)} u &\leq e^{-2R \beta}\left( 1- \frac{e^{2R \beta} - 1}{e^{3R \beta} -1} \right) \inf_{B_R(z)} u \\
    & =  \frac{e^{R \beta} - 1}{e^{3R \beta} -1}  \inf_{B_R(z)} u 
    \end{align}
    for $R < d(z, Y)/4$.

    $$
u(x) -  u(y)e^{-\beta |x-y|} \leq  - \inf_{B_R(z)} u \left(  
\frac{ e^{-\beta d(y)}}{1- e^{-\beta d(y)}}  \left({1- e^{-\beta |x-y|}}\right)
\right) 
    $$

\begin{equation}\label{slope_Harnack}
    \frac{u(x) -  u(y)e^{-\beta |x-y|}}{1- e^{-\beta |x-y|}} \leq  - \inf_{B_R(z)} u \frac{1}{e^{3R \beta}- 1}
\leq - \sup_{B_R(z)} u \frac{1}{e^{R \beta}-1}
\end{equation}
If $u$ is not non-positive, \cref{slope_Harnack} holds with $u$ replaced by $u - \sup_{B_{4R}(z)} u$, so we end up with 
\begin{equation}\label{slope_Harnack}
    \frac{u(x) -  u(y)e^{-\beta |x-y|}}{1- e^{-\beta |x-y|}} 
\leq\sup_{B_{4R}(z)} u+ \left( \sup_{B_{4R}(z)} u - \sup_{B_R(z)} u \right) \frac{1}{e^{R \beta}-1}
\end{equation}
\end{proof}

\begin{lemma}
    Let $\mathcal{F} \subset C(U)$ be a family of functions that enjoy comparison with exponential cones from above in $U$. Suppose 
    $$
    h(x) = \sup_{v \in \mathcal{F}}v(x)
    $$
    is finite and locally bounded above in $U$. Then $h \in C(U)$, and it enjoys comparison with exponential cones from above in $U$.
\end{lemma}

\begin{proof}
    First, choose $v_0 \in \mathcal{F}$ and notice that we may replace $\mathcal{F}$ by

$$
\hat{\mathcal{F}}=\left\{\max \left(v, v_0\right): v \in \mathcal{F}\right\}
$$

without changing $h$. Moreover, $\hat{\mathcal{F}} \subset \operatorname{CECA}(U)$ (see below). We are reduced to the case in which the functions in $\mathcal{F}$ are all locally bounded above (by $h$ ) and below (by $v_0$ ). The continuity of $h$ now follows from Lemma \cref{Harnack_lemma}, which implies that $\mathcal{F}$ is locally equicontinuous. Suppose that $V \subset \subset U$ and $z \notin V$. For $v \in \mathcal{F}$ we have, by assumption and the definition of $h$,

$$
v(x)-a(1-e^{- \beta |x-z|} ) - b e^{- \beta |x-z|} \leq \max _{w \in \partial V}(v(w)- a(1-e^{- \beta |w-z|})  - b e^{- \beta |w-z|} ) 
$$
$$
\leq \max _{w \in \partial V}(h(w)- a(1-e^{- \beta |w-z|}) - b e^{- \beta |w-z|} )
$$
for $x \in V$. Taking the supremum over $v \in \mathcal{F}$ on the left hand side, we find that $h$ enjoys comparison with exponential cones from above.
\end{proof}

\subsection{Local Maxima imply locally constant}
We shall prove the following maximum principle, for functions satistifying the following weak version of \cref{bound_by_slope} :

\begin{equation}\label{bound_by_slope_weak}
u(x) \leq  e^{-\beta |x-y|}u(y)+\max _{\{w:|w-y| \leq r\}}\left( 
\frac{u(w)- e^{-\beta r}u(y)}{1- e^{-\beta r}}
\right)  \left({1- e^{-\beta |x-y|}}\right)
\end{equation}
$\text{ for }  |x-y| \leq r<\operatorname{dist}(y, Y)$.

\begin{lemma}
 Let $U$ be connected, $u \in C(U)$ and \cref{bound_by_slope_weak} hold for $x, y \in U$ with $|x-y| \leq r<\operatorname{dist}(y, Y)$. Let $\hat{x} \in U$ and $u(\hat{x}) \geq u(w)$ for $w \in U$. Then $u \equiv u(\hat{x})$ in $U$.
\end{lemma}
\begin{proof}
The set $\{x \in U: u(x)=u(\hat{x})\}$ is obviously closed in $U$. We show then that it is also open; since $U$ is assumed to be connected, the result follows. It suffices to show that $u$ is constant on some ball containing $\hat{x}$. If $z \in U$ and $|z-\hat{x}| \leq s<$ $\operatorname{dist}(z, Y),$ \cref{bound_by_slope_weak} with $x=\hat{x}, y=z$ and $r=s$ yields

$$
\begin{aligned}
u(\hat{x}) &  \leq \left( e^{-\beta |\hat{x}-z|} - \frac{1-e^{-\beta |\hat{x}-z|}}{e^{\beta s}-1} \right) u(z)
+ \left(\max _{\{w:|w-z| \leq s\}} u(w)\right) \frac{1-e^{-\beta |\hat{x}-z|}}{1-e^{-\beta s}} \\
& \leq \left( e^{-\beta |\hat{x}-z|} - \frac{1-e^{-\beta |\hat{x}-z|}}{e^{\beta s}-1} \right) u(z)
+ u(\hat{x}) \frac{1-e^{-\beta |\hat{x}-z|}}{1-e^{-\beta s}}
\end{aligned}
$$
provided that $u(w) \leq u(\hat{x})$ for $|w-z| \leq s$.

Note that this implies 
$$
\frac{e^{-\beta |\hat{x}-z|} - e^{-\beta s}}{1-e^{-\beta s}} u(\hat{x}) 
\leq 
\frac{e^{\beta (s-|\hat{x}-z|)} - 1}{e^{\beta s} -1 } u(z)
= \frac{e^{-\beta |\hat{x}-z|} - e^{-\beta s}}{1-e^{-\beta s}} u(z)
$$
So 
$$
u(\hat{x}) \leq u(z).
$$
This last condition holds, by assumption, in a neighborhood of $\hat{x}$. Thus for the $z$ 's satisfying all of the above requirements, which clearly cover a neighborhood of $\hat{x}, u(z) \leq u(\hat{x}) \leq u(z)$.

\end{proof}

This implies that if $u$ satisfies \cref{bound_by_slope} (in particular, if $u \in$ $\mathrm{CECA}(U))$ and has a local maximum at some point $\hat{x} \in U$, then it is constant in any connected neighborhood of $\hat{x}$ for which $\hat{x}$ is a maximum point. In particular, \cref{maximum_on_sphere} holds. 

\subsection{Further consequence}

We have shown that 
$$
S^{+}_{\beta,u}(y, r):=\max _{\{w:|w-y|=r\}} \left( \frac{u(w)- e^{-\beta r}u(y)}{1- e^{-\beta r}} \beta \right)
$$
is nondecreasing in $r$ for $0<r<\operatorname{dist}(y, Y)$. And $S^{+}_{\beta,u}(y, r) \geq \beta u(y).$ The limit
\begin{equation}\label{local_slope}
    S^{+}_{\beta,u}(y):=\lim _{r \downarrow 0} S^{+}_{\beta,u}(y, r)=\inf _{0<r<\operatorname{dist}(y, Y)} S^{+}_{\beta,u}(y, r)
\end{equation}

is therefore well-defined and finite. We shall call it "positive slope".

Similarly we define the negative slope:

\begin{equation}\label{local_slope_negative}
    S^{-}_{\beta,u}(y):=\lim _{r \downarrow 0} S^{-}_{\beta,u}(y, r)=\sup _{0<r<\operatorname{dist}(y, Y)} S^{-}_{\beta,u}(y, r)
\end{equation}

We have:

\begin{lemma}
     Let $u \in C(U)$ satisfy (\cref{bound_by_slope}) and $y \in U$.
     \begin{enumerate}
         \item  $S^{+}_{\beta,u}(y) = S^{+}_{0,u}(y) + \beta u(y)$
         \item  $S^{+}_{\beta,u}(y)$, as given by \cref{local_slope}, is upper-semicontinuous in $y \in U$.
         \item If $[w, z] \subset U$, then
         \begin{equation}\label{slope_path_bound}
u(w)- e^{-\beta |w-z| } u(z) \leq \frac{1}{\beta} \left( {\max _{y \in[w, z]} S^{+}_{\beta,u}(y)} 
 \right) (1-e^{-\beta |z-w|})
         \end{equation}
         \item $S^{+}_{\beta,u}(y)=\mathrm{T}^{\beta,+}_u(y)$ where 
         $$
         \mathrm{T}^{\beta,+}_u(y):= \lim_{r \downarrow 0} L^{\beta}_u(B_r(x))
         = \inf \left\{{L}_u^{\beta}\left(B_r(x)\right): 0<r<\operatorname{dist}(x, Y)\right\}
         $$
     \end{enumerate}

Let $u \in C(U)$ satisfy \cref{bound_by_slope_negative} and $y \in U$. The following analogy of negative slope is true:

\begin{enumerate}
         \item  $S^{-}_{\beta,u}(y) = S^{-}_{0,u}(y) + \beta u(y)$
         \item  $S^{-}_{\beta,u}(y)$, as given by \cref{local_slope_negative}, is lower-semicontinuous in $y \in U$.
         \item If $[w, z] \subset U$, then
$$
u(w)- e^{\beta |w-z| } u(z) \geq \frac{1}{\beta} \left( {\min _{y \in[w, z]} S^{-}_{\beta,u}(y)} 
 \right) ( 1- e^{\beta |z-w|})
$$
         \item $S^{-}_{\beta,u}(y)=\mathrm{T}^{\beta,-}_u(y)$ where 
         $$
         \mathrm{T}^{\beta,-}_u(y):= - \lim_{r \downarrow 0} L^{\beta}_u(B_r(x))
         = - \inf \left\{{L}_u^{\beta}\left(B_r(x)\right): 0<r<\operatorname{dist}(x, Y)\right\}
         $$
         
     \end{enumerate}

\end{lemma}

\begin{proof}
    
    Taylor expand the exponential for small $r$:
    $$
    \frac{\beta}{\beta r + o(r)} \left( u(w) - u(y) + \beta r u(y) + o(r) \right)
    $$
    $$
    = \frac{u(w) - u(y)}{ r + o(r)} + \beta \frac{u(y)}{1+o(1)} + \frac{o(r)}{r+ o(r)}
    $$
    Take supremum over $w$ and let $r \to 0$, the first result is thus immediate by lemma \cref{Harnack_lemma}.

    The second part for $\beta = 0$ is proven in \cite{aronsson2004tour}. Thus the general case follows easily with the fact $u$ is continuous. But here is another proof which does not rely on this:

    Note that $\mathrm{T}^{\beta}_u(x)<M$ implies $L^{\beta}_u\left(B_r(x)\right)<M$ for some $r>0$. For $y \in B_r(x)$ and $s+|x-y| \leq r$, we have $B_s(y) \subset B_r(x)$ and then $L^{\beta}_u\left(B_s(y)\right) \leq L^{\beta}_u\left(B_r(x)\right)$. Hence $\mathrm{T}^{\beta}_u(y)<M$. This shows that $\left\{x \in U: \mathrm{T}^{\beta}_u(x)<M\right\}$ is open, establishing the upper-semicontinuity.

    By lemma \cref{Harnack_lemma} if \cref{bound_by_slope} holds, then $u$ is Lipschitz continuous on any compact subset of $U$. Let $[w, z] \subset U$. By the local Lipschitz continuity, $g(t):=u(w+t(z-w))$ is Lipschitz continuous in $t \in[0,1]$. Fix $t \in(0,1)$ and observe that \cref{bound_by_slope} with $x=w+(t+h)(z-w)$ and $y=w+t(z-w)$ implies, for small $h>0$,

$$
\begin{aligned}
\frac{e^{\beta h|z-w|} g(t+h)-g(t)}{e^{\beta h |z-w|} -1} \beta
& =\frac{e^{\beta h |z-w|} u(w+(t+h)(z-w))-u(w+t(z-w))}{ e^{\beta h |z-w|} -1} \beta \\
& \leq S^{+}_{\beta,u} (w+t(z-w), h|z-w|)
\end{aligned}
$$

Let $h \downarrow 0$ we find,

$$
g^{\prime}(t) \frac{1}{ |z-w|} + \beta g(t) \leq 
{\max _{y \in[w, z]} S^{+}_{\beta,u}(y)} 
$$
\begin{equation}\label{derivative_bound}
g^{\prime}(t) \leq 
{\max _{y \in[w, z]} S^{+}_{\beta,u}(y)} |z-w| -  \beta |z-w| g(t)
\end{equation}

If we define $h(t) = e^{\beta t|z-w|} g(t)$, we have 
$$
h^{\prime}(t) \leq e^{\beta t |z-w|}
{\max _{y \in[w, z]} S^{+}_{\beta,u}(y)} |z-w|
$$

Thus 

\begin{equation}
    h(t) - h(0) \leq (e^{\beta t |z-w|} -1)
{\max _{y \in[w, z]} S^{+}_{\beta,u}(y)} \frac{1}{\beta}
\end{equation}

Let $t=1$ we get the desired result.

For the last part, we have

$$
S^{+}_{\beta,u}(y, r)=\max _{|w-y|=r}\left(\frac{u(w)-e^{-\beta r}u(y)}{1 - e^{-\beta r}}\right) \beta \leq L_u^{\beta}\left(B_r(y)\right)
$$

and therefore $S^{+}_{\beta,u}(y) \leq \mathrm{T}_u^{\beta}(y)$. On the other hand, using \cref{slope_path_bound} and the monotonicity of $S^{+}_{\beta,u}(y, r)$ in $r$ and the continuity in $y$, we obtain, if $s>0$ is small,

$$
\begin{aligned}
\mathrm{T}_u^{\beta,+}(y) & =\lim _{r \downarrow 0} L_u^{\beta}\left(B_r(y)\right) \leq \lim _{r \downarrow 0} \sup _{w \in B_r(y)} S^{+}_{\beta,u}(w) \\
& \leq \lim _{r \downarrow 0} \sup _{w \in B_r(y)} S^{+}_{\beta,u}(w, s)=S^{+}_{\beta,u}(y, s) .
\end{aligned}
$$

The remaining inequality, $\mathrm{T}_u^{\beta,+}(y) \leq S^{+}_{\beta,u}(y)$, follows upon sending $s \downarrow 0$. We have proved $S^{+}_{\beta,u}(y)=\mathrm{T}_u^{\beta,+}(y)$.
\end{proof}

\subsection{Increasing slope estimate} 

\begin{lemma}\label{lemma_increasing_slope}
    Let $u \in C(U)$ and \cref{bound_by_slope} hold. If $x^0, x^1 \in U, 0<\left|x^1-x^0\right|<$ dist $\left(x^0, Y\right)$, and

$$
u\left(x^1\right)-e^{-\beta |x^1 - x^0|} u\left(x^0\right) = \frac{1}{\beta} (1-e^{-\beta \left|x^1-x^0\right|}) S^{+}\left(x^0,\left|x^1-x^0\right|\right),
$$

equivalently,

$$
u\left(x^1\right)=\max \left\{ u(w) :\left|w-x^0\right|=\left|x^1-x^0\right|  \right\}
$$

then, for $0<s<\operatorname{dist}\left(x^0, Y\right)-\left|x^1-x^0\right|$,
\begin{equation}\label{increasing_slope_inequality}
    \beta u(x^0) \leq S^{+}\left(x^0,\left|x^1-x^0\right|\right) \leq S^{+}\left(x^1\right) \leq S^{+}\left(x^1, s\right) .
\end{equation}
\end{lemma}

\begin{proof}
    The first inequality just recalls that $\beta u \leq S^{+}_{\beta}$, and the last inequality recalls the monotonicity of $S^{+}_{\beta}$.

First, we use, by assumption, that
$\text{for} \quad\left|x-x^0\right| \leq\left|x^1-x^0\right| ,$
\begin{equation}\label{increasing_slope_1}
    u(x) \leq e^{-\beta |x-x^0|} u\left(x^0\right) + S^{+}\left(x^0,\left|x^1-x^0\right|\right) \frac{1}{\beta} (1-e^{-\beta \left|x-x^0\right|}) \quad 
\end{equation}
Put

$$
x^t=x^0+t\left(x^1-x^0\right) \quad \text { for } \quad 0 \leq t \leq 1
$$

and $x=x^t$ in \cref{increasing_slope_1} to find,

$$
\begin{aligned}
u\left(x^t\right) 
& \leq e^{-\beta t|x^1-x^0|} u\left(x^0\right) + 
S^{+}\left(x^0,\left|x^1-x^0\right|\right) \frac{1}{\beta} (1-e^{-\beta t \left|x^1-x^0\right|}) \\
& =e^{-\beta t|x^1-x^0|} u\left(x^0\right) + 
\left( \frac{u(x^1)- e^{-\beta |x^1-x^0|}u(x^0)}{1- e^{-\beta |x^1-x^0|}}  \right) (1-e^{-\beta t \left|x^1-x^0\right|})
\end{aligned}
$$

Thus

$$
\begin{aligned}
&u\left(x^1\right)- e^{-\beta (1-t) |x^1-x^0|} u\left(x^t\right) \\
& \geq u(x^1) - e^{-\beta |x^1 -x^0|} u(x^0) +(e^{\beta |x^1-x^0|} - e^{\beta(1-t) |x^1-x^0|}) \frac{1}{\beta} S^+(x^0,|x^1-x^0|)
\\
& =( 1-e^{-\beta |x^1-x^0|} + e^{\beta |x^1-x^0|} - e^{\beta(1-t) |x^1-x^0|}) \frac{1}{\beta} S^+(x^0,|x^1-x^0|) \\
& = ( 1 - e^{\beta(1-t) |x^1-x^0|}) \frac{1}{\beta} S^+(x^0,|x^1-x^0|) \\
\end{aligned}
$$
and therefore
\begin{equation}\label{Increaing_slope_2}
    S^{+}\left(x^t,\left|x^1-x^t\right|\right)=S^{+}\left(x^t,(1-t)\left|x^1-x^0\right|\right) \geq S^{+}\left(x^0,\left|x^1-x^0\right|\right) .
\end{equation}

Next we assume that $0<s<\operatorname{dist}\left(x^1, Y\right)$; then

$$
\left|x^t-x^1\right|<s<\operatorname{dist}\left(x^t, Y\right)
$$

if $t$ is near 1, and then, by \cref{Increaing_slope_2} and the monotonicity of $S^{+}$,

$$
S^{+}\left(x^t, s\right) \geq S^{+}\left(x^t,\left|x^1-x^t\right|\right) \geq S^{+}\left(x^0,\left|x^1-x^0\right|\right)
$$

Letting $t \uparrow 1$ in the inequality of the extremes above yields

$$
S^{+}\left(x^1, s\right) \geq S^{+}\left(x^0,\left|x^1-x^0\right|\right)
$$

Letting $s \downarrow 0$, the conclusion follows.

\end{proof}

\section{Proof of the main theorem}
    
We begin the proof of \cref{main}. We first present a one-sided version of \cref{main}:

\begin{proposition}\label{one_side}
Let $u\in C(U)$. Then the following conditions, when imposed for every
$V\Subset U$, are equivalent:
\begin{enumerate}
\item[(a)] If $L\in[0,\infty]$, $z\in\partial V$ and
\[
  u(w)\le e^{-\beta d(w,z)} u(z)+\frac{L}{\beta}(1-e^{-\beta d(w,z)}) \qquad\text{for } w\in\partial V,
\]
then
\[
  u(x)\le e^{-\beta d(w,z)} u(z)+\frac{L}{\beta}(1-e^{-\beta d(w,z)})\qquad\text{for } x\in V.
\]

\item[(b)] If $x\in V$, then
\[
  u(x)\le \Psi\bigl(u|_{\partial V}\bigr)(x).
\]

\item[(c)] If $a\in\mathbb{R}$ and $C(x)$ is a positive exponential cones with center $z$ where $z\notin V$, then for
$x\in V$,
\[
  u(x)-C(x)\le \max_{w\in\partial V}\bigl(u(w)-C(w)\bigr).
\]

\item[(d)] If $v\in C(\overline V)$ satisfies $v=u$ on $\partial V$ and
$v\le u$ in $V$, then
\[
  \sup_{x\in V}\bigl(T_u(x)\bigr)\le \sup_{x\in V}\bigl(T_v(x)\bigr).
\]

The next two conditions are also equivalent to the above and do not
involve the general ``test set'' $V$:

\item[(e)] If $y\in U$, then
\[
  g^{+}(r)=\max_{|w-y|\le r} u(w)
\]
is $\beta$-convex for $0\le r<\operatorname{dist}(y,\partial U)$.

\item[(f)] If $x,y\in U$ and $|x-y|\le r<\operatorname{dist}(y,\partial U)$,
then
\[
  u(x) \leq  e^{-\beta |x-y|}u(y)+\max _{\{w:|w-y| = r\}}\left( 
\frac{u(w)- e^{-\beta r}u(y)}{1- e^{-\beta r}}
\right)  \left({1- e^{-\beta |x-y|}}\right)
\]
\end{enumerate}
\end{proposition}

    Note that \cref{thm_comparison_cones} $\Leftrightarrow$ \cref{thm_biased_infty_lap} $\Leftrightarrow$ \cref{thm_tug_of_war} was proven in \cite{peres2010biased}. \cref{thm_stronglyAM} $\Leftrightarrow$ \cref{thm_variational} are equivalent by a proper definition of $L^{\infty}$ minimizer,as shown in \cite{crandall2003efficient}. And we already have \cref{thm_comparison_cones} $\Rightarrow$  \cref{thm_convexity} by \cref{CECA_convex}.

    \cref{thm_ambe} $\Leftarrow$ \cref{thm_comparison_cones}:
    suppose $u$ satisfies comparison with exponential cones, we first show 
    \begin{equation}\label{Lip_remove_boundary}
    L^\beta (\partial U) = L^\beta (\partial (U\setminus \{x \})) = 
    L^\beta (\partial U \cup \{x \})
    \end{equation}
    $$
    $$

    


        It then suffices to show, for any $x \in U, y\in \partial U$, 
    \begin{equation}\label{compar_beta_extend_11}
        e^{\beta |x-y|} u(y) - L \frac{(e^{\beta |x-y|}- 1)}{\beta}
    \leq    u(x)  
    \end{equation}

    \begin{equation}\label{compar_beta_extend_22}
        u(x) \leq  e^{-\beta |x-y|} u(y)  + L \frac{(1- e^{-\beta |x-y|})}{\beta}
    \end{equation}

    Note that \cref{compar_beta_extend_11} and $\cref{compar_beta_extend_22}$ are true when $x \in \partial U$, by comparison with exponential cones, they are also true for $x \in U$: comparison from below implies  \cref{compar_beta_extend_11} (negative cone), and comparison from above implies  \cref{compar_beta_extend_22} (positive cone).

    Thus if we use \cref{Lip_remove_boundary} twice, we get
    $$
    L^\beta (\partial U) = L^\beta (\partial (U\setminus \{x,y \})) = 
    L^\beta (\partial U \cup \{x ,y\})
    $$
    Therefore 
    $$
    u(y) - e^{-\beta |x-y|} u(x)  \leq  L \frac{(1- e^{-\beta |x-y|})}{\beta}
    $$
    and 
    $$
    u(x) - e^{-\beta |x-y|} u(y)  \leq  L \frac{(1- e^{-\beta |x-y|})}{\beta}
    $$
    So $L^{\beta}(U) = L^{\beta}(\partial U)$, $u$ is AM.

\bigskip 

    \cref{thm_ambe} $\Rightarrow$ \cref{thm_comparison_cones}:  Suppose now that $u \in \operatorname{AM}(U)$. We want to show that $u$ then enjoys comparison with cones from above. Assume that $V \subset \subset U, y \notin V$ and set

\begin{equation}
\begin{aligned}
W := \Bigl\{\, x \in V :\;
&u(x) - e^{-\beta |x - y|} u(y)
    - A \bigl(1 - e^{-\beta |x - y|}\bigr) \\
&> \max_{w \in \partial V}
   \bigl( u(w) - e^{-\beta |w - y|} u(y)
   - A \bigl(1 - e^{-\beta |x - y|}\bigr) \bigr)
   \Bigr\}.
\end{aligned}
\end{equation}

We want to show that $W$ is empty. If it is not empty, then it is open and

\begin{equation}
\begin{aligned}
u(x) &= e^{-\beta |x - y|} u(y)
    + A \bigl(1 - e^{-\beta |x - y|}\bigr) \\
&\quad + \max_{w \in \partial V}
    \Bigl( u(w) - e^{-\beta |w - y|} u(y)
    - A \bigl(1 - e^{-\beta |x - y|}\bigr) \Bigr)
    =: C(x) \quad \text{for } x \in \partial W.
\end{aligned}
\end{equation}

Therefore $u=C$ on $\partial W$ and $L_u(W)=L_C(\partial W)$ since $u \in \mathrm{AM}(U)$. But then $u=C$ in $W$ by \cref{equality_cone_on_boundary}, which is a contradiction.

\bigskip


We state the following proposition and postpone its proof until the end of the section. 
\begin{proposition}\label{CEC_inner_set}
Let $u \in C(U)$. Assume that for each $x \in U$ there is a neighborhood $V \subset \subset U$ of $x$ such that $u \in \operatorname{CECA}(V)$. Then $u \in \operatorname{CECA}(U)$.
\end{proposition}

\begin{proposition}
    Let $U$ be bounded and $u \in C(\bar{U})$ satisfy \cref{bound_by_slope}, $x^0 \in U, S^{+}_0\left(x^0\right)>0$ and $\delta>0$. Then there is a sequence of points $\left\{x^j\right\}_{j=1}^{M} \subset U$ and a point $x^{M} \in \partial U$ with the following properties:

    \begin{enumerate}
        \item $\left|x^j-x^{j-1}\right| \leq \delta \text { for } j=1,2, \ldots .$
        \item $\left[x^{j-1}, x^j\right] \subset U$ for $j=1,2, \ldots$.
        \item $S^{+}\left(x^j\right) \geq S^{+}\left(x^{j-1}\right)$ for $j=1,2, \ldots$.
        \item $u\left(x^{\infty}\right)- e^{-\beta \sum_{j=0}^{\infty} \left|x^j-x^{j-1}\right|} u\left(x^0\right) \geq \frac{1}{\beta} S^{+}\left(x^0\right) (1-e^{-\beta \sum_{j=0}^{\infty} \left|x^j-x^{j-1}\right|})$.
    \end{enumerate}
\end{proposition}
\begin{proof}
    Let $x^0 \in U$ and $S^{+}_0\left(x^0\right)>0$. Iteratively define $x^j$, $j=1,2, \ldots$ so that

$$
x_j = \arg \max_{w} \{u(w) : |w-x_{j-1}| \leq \delta \}
$$

by \cref{lemma_increasing_slope}

$$
\frac{u\left(x^j\right)- e^{-\beta \left|x^j-x^{j-1}\right|} u\left(x^{j-1}\right)}{1-e^{-\beta \left|x^j-x^{j-1}\right|}}
= \frac{1}{\beta} S^{+}\left(x^{j-1},\left|x^j-x^{j-1}\right|\right)
$$
and
$$
S^{+}\left(x^{j-1},\left|x^j-x^{j-1}\right|\right) \geq S^{+}\left(x^{j-2},\left|x^{j-1}-x^{j-2}\right|\right)
$$
for $j=2,3, \ldots$. Thus, by the construction,

$$
\begin{aligned}
u\left(x^j\right)- e^{-\beta \left|x^j-x^{j-1}\right|} u\left(x^{j-1}\right) 
& =(1 - e^{-\beta \left|x^j-x^{j-1}\right|}) \frac{1}{\beta} S^{+}\left(x^{j-1},\left|x^j-x^{j-1}\right|\right) \\
& \geq (1-e^{-\beta \left|x^j-x^{j-1}\right|}) \frac{1}{\beta} S^{+}\left(x^0,\left|x^1-x^0\right|\right) \\
& \geq (1-e^{-\beta \left|x^j-x^{j-1}\right|}) \frac{1}{\beta} S^{+}\left(x^0\right)
\end{aligned}
$$

Note that 

$$
u\left(x^N\right)- e^{-\beta \sum_{j=1}^N\left|x^j-x^{j-1}\right|} u\left(x^{0}\right)
=
\sum_{j=1}^{N} e^{-\beta \sum_{k=j+1}^{N} |x^k -x^{k-1}|} \left( u\left(x^j\right)- e^{-\beta \left|x^j-x^{j-1}\right|} u\left(x^{j-1}\right) \right)
$$
$$
\geq \frac{1}{\beta} S^{+}\left(x^0\right) (1-e^{-\beta \sum_{j=1}^{N} \left|x^j-x^{j-1}\right|})
$$


Now the only remaining question is: will there exist $N$ such that $x_N \in \partial U$? Note that $x_j \not \in B_{\delta}(x_{j-2})$ by our choice of $x_j$'s. Since $\bar{U}$ is compact, thus there is finite cover of it by 

Similarly $x_j$ is not in $B_\delta\left(x_k\right)$ for any $j \geq k+2$. So, for every $k \neq j$ we have $d\left(x_{2 k}, x_{2 j}\right) \geq \delta$, i.e., the open balls $B_{\delta / 2}\left(x_{2 k}\right)$ are all disjoint. Now, since $\bar{U}$ is compact, there exists a finite covering of it by open $\delta / 4$-balls. Each ball $B_{\delta / 2}\left(x_{2 k}\right)$ contains one of these $\delta / 4$-balls entirely, hence there are more $\delta / 4$-balls in the covering than points $x_{2 k}$. Therefore, there is a uniform finite upper bound $M=M(\delta)$ such that $N \leq M$, which is independent of the choice of $x_0$.

\end{proof}

Proof of (d,f):
We prove in \cref{one_side} $(f) \;\Rightarrow\; (d)$. 
Assume that $(f)$ holds but $(d)$ does not. 
Then there are $V \Subset U$, $v \in C(\overline{V})$ and $x^{0} \in V$ for which
\[
u \ge v \quad \text{in } V, 
\qquad
u = v \quad \text{on } \partial V,
\qquad
T_u(x^{0}) > \sup_{x \in V} T_v(x).
\]

\[
\begin{aligned}
u(x^\infty)- e^{-\beta \sum_{j=1}^{\infty}\lvert x^{j}-x^{j-1}\rvert}u(x^0)
&\ge \frac{S^{+}(x^0)}{\beta}  (1-e^{-\beta \sum_{j=0}^{\infty} \left|x^j-x^{j-1}\right|})  \\[6pt]
&> \frac{\sup_{V}\, \mathrm{T}_{v}\!}{\beta} \left( 1-e^{-\beta \sum_{j=0}^{\infty} \left|x^j-x^{j-1}\right|} \right)
   \; \\
    &= \frac{\sup_{V}\, \mathrm{T}_{v}\!}{\beta} \sum_{j=1}^{\infty} e^{-\beta \sum_{k=j+1}^{\infty} |x^k -x^{k-1}|} 
    \left( 1- e^{-\beta \left|x^j-x^{j-1}\right|} \right) \\[6pt]
& \ge \sum_{j=1}^{\infty} e^{-\beta \sum_{k=j+1}^{\infty} |x^k -x^{k-1}|} \left( v\left(x^j\right)- e^{-\beta \left|x^j-x^{j-1}\right|} v\left(x^{j-1}\right) \right)
   \; \\
   & =v(x^\infty)- e^{-\beta \sum_{j=1}^{\infty}\lvert x^{j}-x^{j-1}\rvert}v(x^0).
\end{aligned}
\]

Since $u = v$ on $\partial V$, $u(x^\infty) = v(x^\infty)$, and the above implies
$u(x^0) < v(x^0)$, contradiction.

\begin{proof}[Proof of \cref{CEC_inner_set}]
    Under the assumptions of \cref{one_side}, if $y \in U$, the function

$$
g^{+}(r)=\max _{\{w:|w-y| \leq r\}} u(w)
$$

will be $\beta$-convex for $0 \leq r \leq \delta$ where $\delta$ is a (possibly small) positive number. This is because $u \in \operatorname{CECA}\left(B_\delta(y)\right)$ for some $\delta>0$, thus by \cref{CECA_convex}. Let $R$ be the largest number satisfying $0<R \leq \operatorname{dist}(y, \partial U)$ such that $g^{+}(r)$ is $\beta$-convex on $[0, R)$. If $R=\operatorname{dist}(y, \partial U)$, we are done. Assuming that $R<\operatorname{dist}(y,  \partial U)$, we will derive a contradiction. By compactness and the assumptions, there is some number $0<\kappa<\operatorname{dist}(y,  \partial U)-R$ such that for $|w-y|=R$

$$
g_w^{+}(r)=\max _{\{z:|z-w| \leq r\}} u(z)
$$

is $\beta$-convex on $[0, \kappa]$. Then
\begin{equation}\label{4.12}
    g^{+}(R+s)=\max _{\{w:|w-y|=R\}} g_w^{+}(s)
\end{equation}

for $0 \leq s \leq \kappa$. As the supremum of $\beta$-convex functions is $\beta$-convex (one can easily check this), $g^{+}(R+s)$ is $\beta$-convex in $s$ for $0 \leq s \leq \kappa$. One easily sees that $g^{+}(r)$ is then $\beta$-convex on $0 \leq r \leq R+\kappa$ if and only if the left derivative of $g^{+}(r)$ at $r=R$ is less than or equal to the right derivative of $g^{+}(r)$ at $r=R$. But if $w,|w-y|=R$, is chosen so that $g^{+}(R)=u(w)$, the right derivative of $g_w^{+}(s)$ at 0 , which is less than or equal to the right derivative of $g^{+}(r)$ at $r=R$ by \cref{4.12}, enjoys the desired estimate by (the proof of) \cref{lemma_increasing_slope}. Using the equivalence of (c) and (e), we are done.

\end{proof}

\begin{proof}[Proof of \cref{main}]
    As we have proven the one-sided version \cref{one_side}, the conclusion follows with similar argument in the other direction.
\end{proof}

\section{Biased tug-of-war on graphs}
\subsection{Biased Tug-of-war on graphs}

In the unbiased case, the AMLE is unique as long as the boundary data $g$ is Lipschitz and bounded. In the biased case, an example in \cite{peres2010biased} shows this is in general not true. However, the example given there does not satisfy the maximum principle: the maximum is not obtained on the boundary. One might ask if there is an unique solution satisfying the maximum principle. The following example shows this is not the case: 

Let $\rho = 2$. Consider a graph with $V = \{ (0,1),(-1,0), (0,0), (1,0), (2,0), (3,0), \dots  \}$, i.e. it contains $(x,0)$ with $x\geq-1$ and $(0,1)$. Let $Y = \{(0,1),(-1,0) \}$ and $g( (0,1) ) = 1$ and $g( (0,1) ) = -2$. 

One can check that 
$u((0,0))=0$, $g((x,0))= a (1-\frac{1}{2^x})$, for $a \in (0,1)$ is a solution satisfying the maximum principle.

Thus in general one cannot expect the value of game exists, and uniqueness of \cref{discrete_biased_infty} on unbounded domains.

\subsection{Continuum value of biased tug-of-war on length space}
When the underline space is compact, this is well established by \cite{peres2010biased}. Here we improve some estimates. To prove the convergence of continuum value, one needs to introduce the favored game introduced by \cite{peres2008tug, peres2010biased}. We recall the definition for readers' convenience. 

The \emph{biased II-favored $\varepsilon$-tug-of-war} is defined as follows. At the $k^{\text{th}}$ step,
player I chooses $z_k \in B_{\varepsilon}(x_{k-1})$. If player I wins the coin toss, then player II can choose
$x_k \in (B_{2\varepsilon}(z_k)\cap Y)\cup \{z_k\}$. In other words, he moves to $z_k$ or terminates the game. If player
II wins the coin toss he moves to a position in $B_{2\varepsilon}(z_k)$ of his choice.

The value to player I of the biased II-favored $\varepsilon$ tug-of-war is denoted by $v^\varepsilon$.
The value to player II of the correspondingly defined \emph{biased I-favored $\varepsilon$ tug-of-war game} is denoted
by $w^\varepsilon$. In both cases, the bias is $\frac{1+\theta(\varepsilon)}{2}$, favoring player I.

As the names suggest, player I cannot do better in the II-favored game than in the ordinary
game. In other words, $v^\varepsilon \le u_I^\varepsilon$. This is because given any strategy for player II in the ordinary
$\varepsilon$-game, player II can follow it also in the II-favored game, since for any $z \in B_{\varepsilon}(x)$ we have
$B_{\varepsilon}(x) \subseteq B_{2\varepsilon}(z)$. Similarly, player II cannot do better in the I-favored game than in the
ordinary game (he cannot ensure a value that is smaller than the value in the ordinary game),
so $u_{II}^\varepsilon \le w^\varepsilon$. Hence we have
\[
v^\varepsilon \le u_I^\varepsilon \leq u_{II}^\varepsilon \le w^\varepsilon.
\]
Moreover, the following lemma holds:

\begin{lemma}[\cite{peres2010biased}, Lemma 3.1]
If $\rho(\cdot)$ is log-concave, then, for any $\varepsilon > 0$,
\[
v^{2\varepsilon} \;\le\; v^{\varepsilon} \;\le\; u^{\varepsilon}_I \leq  u^{\varepsilon}_{II}\;\le\; w^{\varepsilon}.
\]

If $\rho(\cdot)$ is log-convex, then
\[
v^{\varepsilon} \;\le\; u^{\varepsilon}_{I} \leq u^{\varepsilon}_{II} \;\le\; w^{\varepsilon} \;\le\; w^{2\varepsilon}.
\]
\end{lemma}

The following is an improvement of Lemma 3.2 in \cite{peres2010biased}

\begin{lemma}\label{II-favored_estimate}
Let $\varepsilon > 0$. Then for each $x \in X$ and $y \in Y$,
\[
v^{\varepsilon}(x) \;\ge\; e^{\beta (2\varepsilon + d^{\varepsilon}(x,y))} g(y) - (e^{\beta (2\varepsilon + d^{\varepsilon}(x,y))} - 1) \frac{L^{\beta}}{\beta}.
\]
Such an expected payoff is guaranteed for player I if he adopts a ``pull towards $y$'' strategy
which at each move attempts to reduce $d^{\varepsilon}(x_k, y)$.
\end{lemma}

\begin{proof}
Given $x = x_0 \in X$ and $y \in Y$, the ``pull towards $y$'' strategy for player I in the
II-favored game ensures that the game ends almost surely in finite time $\tau < \infty$, and that
$e^{\beta d^{\varepsilon}(x_n, y)}$ is a supermartingale, except at the last step where player~II may have moved the
game position up to $\varepsilon$ farther from $y$ even if player~I won the coin toss. This implies
\[
\mathbb{E}\big[ e^{ \beta d^{\varepsilon}(x_{\tau}, y)}\big]
   < e^{ \beta (d^{\varepsilon}(x, y) + 2\varepsilon)} 
\]

Since $g$ has bounded $\beta$-slope on $Y$, we get
\[
\mathbb{E}\big[g(x_{\tau})\big]
   \ge e^{\beta (2\varepsilon + d^{\varepsilon}(x,y))} g(y) - (e^{\beta (2\varepsilon + d^{\varepsilon}(x,y))} - 1) \frac{L^{\beta}}{\beta}.
\]

Since $v^{\varepsilon}$ is at least as great as the expected payoff under this pull toward $y$ strategy,
the lemma is proved. \qedhere
\end{proof}

The next is an improvement of Lemma 3.3 in \cite{peres2010biased}, notably without any assumption on the diameter of $X$, and the proof is simple and intuitive due to the absolute minimizing point of view.
\begin{lemma}\label{u_I_estimate_upper}
If $x \in X$, $y \in Y$ and $d^{\varepsilon}(x,y) = \varepsilon$, then
\[
u^{\varepsilon}_I(x)
\;\le\; e^{-\beta d^{\varepsilon}(x, y)} g(y) + \frac{L^{\beta}}{\beta}  (1-e^{-\beta  d^{\varepsilon}(x_{}, y) })
\]
\end{lemma}

\begin{proof}
    Consider the ``pull away from $y$'' strategy for player II. In this case $e^{-\beta d^{\varepsilon}(x_n, y)}$ is a submartingale. Consider any strategy for Player I to end the game almost surely, we have 
    $$\E g (x_{\tau}) \leq \E e^{-\beta d^{\varepsilon}(x_{\tau}, y)} g(y) + \frac{L^{\beta}}{\beta}  (1-e^{-\beta  d^{\varepsilon}(x_{\tau}, y) }) \leq  e^{-\beta d^{\varepsilon}(x, y)} g(y) + \frac{L^{\beta}}{\beta}  (1-e^{-\beta  d^{\varepsilon}(x_{}, y) })$$
\end{proof}

\begin{proposition}
    $\left\|v^\epsilon-u_I^\epsilon\right\|_{\infty}=O(\epsilon) .$
\end{proposition}
\begin{proof}
    Again when $X$ is compact this is proven in \cite{peres2010biased} proposition 3.5. In the general case, this is proven line by line by the same arguement, except we need estimate that does not reply on the diameter of $X$, which is just established in \cref{u_I_estimate_upper,II-favored_estimate}.
\end{proof}

Due to the monotonicity of $v^{\epsilon}$, we have immediately 

\begin{corollary}
    the following
limits exist pointwise :
\[
\tilde{u}_{\varepsilon,\rho}
\;:=\;
\lim_{n \to \infty} v^{\varepsilon / 2^{n}}
\;=\;
\lim_{n \to \infty} u_I^{\varepsilon / 2^{n}}
\]

\end{corollary}

\subsection{Equivalent-CEC-tug-of-war}

\begin{lemma}
Let $\varepsilon > 0$, let $V$ be an open subset of $X \setminus Y$ and write 
\[ V_\varepsilon = \{x : \overline{B_\varepsilon(x)} \subset V \}. \]

\begin{enumerate}
    \item Suppose that $\varphi(x) = Q(d(x,z))$ is a positive exponential cone. If the value function $u_I^\varepsilon$ for player I in $\varepsilon$-tug-of-war satisfies 
$u_I^\varepsilon \le \varphi$ on $V \setminus V_\varepsilon$, then 
$u_I^\varepsilon \le \varphi$ on $V_\varepsilon$. 
    \item Suppose that $\varphi(x) = Q(d(x,z))$ is a negative exponential cone. If the value function $u_{II}^\varepsilon$ for player II in $\varepsilon$-tug-of-war satisfies 
$u_{II}^\varepsilon \geq \varphi$ on $V \setminus V_\varepsilon$, then 
$u_{II}^\varepsilon \geq \varphi$ on $V_\varepsilon$. 
    \item Moreover, suppose that $\varphi(x) = Q(d(x,z))$ is a negative exponential cone. If the value function $u_{I}^\varepsilon$ for player I in $\varepsilon$-tug-of-war satisfies 
$u_{I}^\varepsilon \geq \varphi$ on $V \setminus V_\varepsilon$, then 
$u_{I}^\varepsilon \geq \varphi$ on $V_\varepsilon$. 
\end{enumerate}

\end{lemma}

\begin{proof} The proof is essentially similar to \cite{PSSW} Lemma 3.4. We first prove comparison from above for $u_I^{\epsilon}$: 

Fix some $\delta > 0$. Consider the strategy for player II that from a state 
$x_{k-1} \in V_\varepsilon$ at distance $r = d(x_{k-1},z)$ from $z$ pulls to state $z$ (if $r < \varepsilon$) 
or else moves to reduce the distance to $z$ by ``almost'' $\varepsilon$ units, enough to ensure that 
$Q(d(x_k,z)) < Q(r - \varepsilon) + \delta 2^{-k}$. If $r < \varepsilon$, then $z \in V$. In this case, if II wins the toss, then $x_k = z$, whence 
$\varphi(x_k) = Q(0) \le Q(r - \varepsilon)$. Thus for all $r \ge 0$, regardless of what strategy player I adopts,
\[
\mathbb{E}[\varphi(x_k) \mid x_{k-1}] - \delta 2^{-k-1} 
    \le \frac{e^{\beta \epsilon} Q(r + \varepsilon) + Q(r - \varepsilon)}{1+e^{\beta \epsilon}}
    = Q(r) 
    = \varphi(x_{k-1})  .
\]
Setting $\tau := \inf\{k : x_k \notin V_\varepsilon\}$, we conclude that 
$M_k := \varphi(x_{k \wedge \tau}) + \delta 2^{-k}$ 
is a supermartingale.

Suppose that player I uses a strategy with expected payoff larger than $-\infty$. 
(If there is no such strategy, the assertion of the Lemma is obvious.) Then $\tau < \infty$ a.s.  
$M_k$ is a supermartingale bounded below, therefore
\begin{equation}\tag{3.3}
    \mathbb{E}[M_\tau] \le M_0 .
\end{equation}

Since $u_I^\varepsilon(x_\tau) \le \varphi(x_\tau)$, we deduce that
\[
u_I^\varepsilon(x_0) \le \sup_{S_I} \mathbb{E} 
\Big[ \varphi(x_\tau) \Big]
\le \sup_{S_I} \mathbb{E}[M_\tau] 
\le M_0 = \varphi(x_0) + \delta .
\]

where $S_I$ runs over all possible strategies for player I with expected payoff larger than $-\infty$. 
Since $\delta > 0$ was arbitrary, the proof is now complete. 

\bigskip 

For comparison from below for $u_{II}^{\epsilon}$, the argument is essentially the same by letting Player I to move toward $z$. This this case Player II may not have a strategy to end the game almost surely, so $u_{II} =\infty$ and the statement is trivially true. Otherwise, the argument is the same as above.

\bigskip

For comparison from below for $u_{I}^{\epsilon}$, again we let Player I to move toward $z$. If $z \in Y$, the game must terminates in finite time and the argument is the same as before. If not, consider the strategy for player I after it reaches $z$, such that the game terminates almost surely and the final payoff is at least $u^{\epsilon}_I (z) - \delta$. Note that such strategy exists by definition and $u^{\epsilon}_I (z) > - \infty$ since player I always have a strategy to end the game almost surely. Let $\tau$ be the time that the game terminates and $T_z$ be the first hitting time of $z$, and $\tau_k = \min \{T_z, \tau \}$.

Thus under this strategy 
\begin{align*}
    u_I^\varepsilon(x_0) &\geq
\inf_{S_{II}} \mathbb{E} 
\Big[ u(x_\tau) \Big] \\
&  = \inf_{S_{II}} \mathbb{E} 
\Big[ u(x_\tau) 1_{\tau_z \leq \tau} + u(x_\tau)1_{\tau_z > \tau}  \Big]
\\
&  
\geq \inf_{S_{II}} \mathbb{E} 
\Big[ (u_I^{\epsilon}(z)-\delta) 1_{\tau_z \leq \tau} + u(x_\tau)1_{\tau_z > \tau}  \Big]
\\
&  \geq \inf_{S_{II}} \mathbb{E} 
\Big[ u_I^{\epsilon}(z) 1_{\tau_z \leq \tau} + u_I^{\epsilon}(x_\tau)1_{\tau_z > \tau} - \delta  \Big]
\\
&  \geq \inf_{S_{II}} \mathbb{E} 
\Big[ \varphi(z) 1_{\tau_z \leq \tau} + \varphi(x_\tau)1_{\tau_z > \tau} - \delta  \Big]
\\
&  = \inf_{S_{II}} \mathbb{E} \varphi (x_{\tau_z}) - \delta
\\
&  \geq \inf_{S_{II}} \mathbb{E}[M_{\tau_z}] -\delta
\\
&  \geq M_0 -\delta = \varphi(x_0) .
\\
\end{align*}
Note that the last step is by optional stopping theorem since $M_{t \wedge \tau_z}$ is a sub martingale bounded above. Thus the statement follows.

\qedhere
\end{proof}
Note that when $u^{\epsilon}_I \not = u^{\epsilon}_{II}$, the above shows $u^{\epsilon}_I$ still satisfies CECB, since $\beta>0$ and player I always has some strategies to end the game almost surely. This argument fails if one tries to show CECA for $u^{\epsilon}_{II}$. Thus we have the following corollary.

\begin{corollary}
     We have the
uniform limits :
\[
\tilde{u}_{\varepsilon,\rho}
\;:=\;
\lim_{n \to \infty} v^{\varepsilon / 2^{n}}
\;=\;
\lim_{n \to \infty} u_I^{\varepsilon / 2^{n}}
\]

\end{corollary}

\begin{lemma}[Uniform Lipschitz estimate]
Let $x_1, x_2 \in X \setminus Y$.
Recall $m = \inf_Y F$.
Then, for any $\varepsilon > 0$, we have
\[
\lvert u_I^{\varepsilon}(x_1) - u_I^{\varepsilon}(x_2) \rvert
\;\le\;
({L^{\beta}} - \beta m) d^{\varepsilon}(x_1, x_2).
\]
\end{lemma}

\begin{proof}
    As we have shown $u^{\epsilon}_I$ satisfies comparison with exponential cones, by \cref{CECA_convex}, we have 
    $$
    u_I^{\epsilon}(x_1) - e^{-\beta d^{\varepsilon}(x_1, x_2) } u_I^{\epsilon}(x_2) \leq \frac{L^{\beta}}{\beta} (1-e^{-\beta d^{\varepsilon}(x_1, x_2)})   $$
    $$
    u_I^{\epsilon}(x_1) -  u_I^{\epsilon}(x_2) \leq 
    (\frac{L^{\beta}}{\beta} -u^{\epsilon}_I(x_2)) (1-e^{-\beta d^{\varepsilon}(x_1, x_2)}) \leq ({L^{\beta}} - \beta u^{\epsilon}_I(x_2)) d^{\varepsilon}(x_1, x_2).  $$
    Thus the result follows immediately.
\end{proof}

\begin{theorem}
    Any subsequential uniform limit $\tilde{u}=\lim _{n \rightarrow \infty} u_I^{\epsilon_n}$ satisfies comparison with $\beta$-exponential cones in $X$, and hence it is continuous on $X$.
\end{theorem}

\subsection{Uniqueness}
Again we shall follow the idea in \cite{peres2010biased} and upgrade the argument in the general setting.

\begin{lemma}
    Let $v : X \to \mathbb{R}$ be continuous and satisfy comparison with $\beta$-exponential
cones from above on $X \setminus Y$. Suppose that $\theta(\varepsilon)$ satisfies
\[
\frac{1}{1 + e^{-\beta \varepsilon}} \;\le\; \frac{1+\theta(\varepsilon)}{2},
\]
and fix $\delta > 0$. Then, in II-favored tug-of-war, player~II may play to make
$M_{k \wedge \tau}$ a supermartingale, where
\[
M_k := v(x_k) + \delta/2^{k}
\quad\text{and}\quad
\tau := \inf\{ k : d(x_k, Y) < 3\varepsilon \}.
\]
\end{lemma}

\begin{proof}
    The proof is identical to \cite{peres2010biased} Lemma 5.1 in the general case.
\end{proof}

\begin{lemma}
    Suppose that $v : X \to \mathbb{R}$ is continuous, satisfies comparison with $\beta$-exponential
cones from below on $X \setminus Y$, and $v \ge F$ on $Y$.
Let $\theta(\varepsilon)$ satisfy
\[
\frac{1}{1 + e^{-\beta \varepsilon}} \;\ge\; \frac{1+\theta(\varepsilon)}{2}.
\]
Then $v^{\varepsilon} \le v$ for all $\varepsilon > 0$, where $v^{\varepsilon}$ is the value of the
II-favored $\varepsilon$-game for player~I.

\end{lemma}

\begin{proof}
    The proof is almost identical with \cite{peres2010biased} lemma 5.2, except we need another argument for comparison with exponential cones as the original proof relies on $X$ is compact. However this is easy with $L^{\beta} < \infty$: Let $a = L^{\beta}$ and $C(r):=a\left(1-e^{-\beta r}\right)+ e^{-\beta r} g(y)$. The reader can easily check that this gives the desired exponential cone.
\end{proof}

\bigskip

\begin{proof}[Proof of \cref{u_I_theorem} and \cref{main_general}]
    From above we have shown that if $u$ is an extension that satisfies CECA, then playing the favored $\epsilon$-games with $\theta_0(\epsilon)=\tanh (\beta \epsilon / 2)$,  $v^\epsilon \leq u$ for any $\epsilon>0$, hence we have $\tilde{u} \leq u$ given any such $u$. Thus $\tilde{u}$ is the smallest of such $u$. As we have shown $\tilde{u}$ itself satisfies CECA, this completes the proof.
\end{proof}

\section{Linear blow up}

\begin{proof}[Proof of \cref{linear-blow-up}]
By subtracting a constant we may consider $u(x^0) = 0$. 

Note we have shown $S_{\beta,u}^{+} (x) = S_{0,u}^{+} (x) + \beta u(x)$. When $u(x) = 0$, the $\beta$-slope is the same as the local Lipschitz constant. For this reason we shall omit the $\beta$ in the notation in the following, when there is no confusion.

We will assume that

$$
L_0^{\beta}:=S_u^{+}\left(x^0\right)=\mathrm{T}_u^{\beta}\left(x_0\right)>0
$$
as the case $L_0^{\beta} = 0$ follows easily by the upper-semicontinuity of $L^{\beta}$.

If $u \in \mathrm{CEC}(U), x^0 \in U, B_{r_0}\left(x^0\right) \subset \subset U$ and $\tilde{u}(x)=\left(u\left(r_0 x+x^0\right) \right) /\left(r_0 L_0\right)$, then $\tilde{u} \in \operatorname{CEC}\left(B_1(0)\right), \tilde{u}(0)=0$ and $S_{\tilde{u}}^{+}(0)=S_u^{+}\left(x^0\right) / L_0=1$. Hereafter, we simply assume that

$$
u \in \mathrm{CEC}\left(B_1(0)\right), \quad u(0)=0, \quad S_u^{+}(0)=1
$$

Given above, for $\lambda>0$ the function

$$
v_\lambda(x):=\frac{u(\lambda x)}{\lambda}
$$

satisfies $v_\lambda \in \operatorname{CEC}\left(B_{1 / \lambda}(0)\right), v_\lambda(0)=0$ and for $r<1 / \lambda$,

$$
L_{v_\lambda}\left(B_r(0)\right)=L_u\left(B_{\lambda r}(0)\right), \quad \max _{|w|=r} v_\lambda(w)=\frac{\max _{|w|=\lambda r} u(w)}{\lambda}=r S_u^{+}(0, \lambda r)
$$

Thus the family $v_\lambda$ is uniformly bounded and equicontinuous in each ball $B_r(0)$ as $\lambda \downarrow 0$. Therefore there exists a sequence $\lambda_j \downarrow 0$ and $v \in C\left(\mathbb{R}^n\right)$ such that $v_{\lambda_j} \rightarrow v$ uniformly on every bounded set. Clearly $v \in \mathrm{CEC}\left(\mathbb{R}^n\right)$ (comparison with exponential cones is obviously preserved under uniform convergence). Putting $\lambda=\lambda_j$ in the relations above and passing to the limit then yields the first two claims below:

$$
L_v\left(B_r(0)\right) \leq \mathrm{T}_u(0)=1, \quad \max _{|w|=r} v(w)=r, \quad \min _{|w|=r} v(w)=-r
$$

By the above property, the result follows the same way as in \cite{aronsson2004tour}.
\end{proof}

\subsection*{Acknowledgement}
This work is supported by NSF-DMS grants 2153359 and 2450726. The author would like to thank Alan Hammond for introducing the topic on tug-of-war to the author, encouragement and generousness sponsorship during this work.

\bibliographystyle{plain}
\bibliography{bibliography}

@article{PSSW,
  title={Tug-of-war and the infinity Laplacian},
  author={Peres, Yuval and Schramm, Oded and Sheffield, Scott and Wilson, David},
  journal={Journal of the American Mathematical Society},
  volume={22},
  number={1},
  pages={167--210},
  year={2009}
}

@article{peres2008tug,
  title={Tug-of-war with noise: a game-theoretic view of the p-Laplacian},
  author={Peres, Yuval and Sheffield, Scott},
  journal={Duke Math J.},
  year={2008}
}

@article{peres2010biased,
  title={Biased tug-of-war, the biased infinity Laplacian, and comparison with exponential cones},
  author={Peres, Yuval and Pete, G{\'a}bor and Somersille, Stephanie},
  journal={Calculus of Variations and Partial Differential Equations},
  volume={38},
  number={3},
  pages={541--564},
  year={2010},
  publisher={Springer}
}

@article{mcshane1934extension,
  title={Extension of range of functions},
  author={McShane, Edward James},
  journal={Bulletin of the American Mathematical Society},
  year={1934}
}

@article{Whitney34,
 ISSN = {00029947, 10886850},
 URL = {http://www.jstor.org/stable/1989708},
 author = {Hassler Whitney},
 journal = {Transactions of the American Mathematical Society},
 number = {1},
 pages = {63--89},
 publisher = {American Mathematical Society},
 title = {Analytic Extensions of Differentiable Functions Defined in Closed Sets},
 urldate = {2025-02-23},
 volume = {36},
 year = {1934}
}

@article{aronsson2004tour,
  title={A tour of the theory of absolutely minimizing functions},
  author={Aronsson, Gunnar and Crandall, Michael and Juutinen, Petri},
  journal={Bulletin of the American mathematical society},
  volume={41},
  number={4},
  pages={439--505},
  year={2004}
}

@article{champion2007principles,
  title={Principles of comparison with distance functions for absolute minimizers},
  author={Champion, Thierry and De Pascale, Luigi and others},
  journal={Journal of Convex Analysis},
  volume={14},
  number={3},
  pages={515},
  year={2007},
  publisher={HELDERMANN VERLAG LANGER GRABEN 17, 32657 LEMGO, GERMANY}
}

@article{evans2011everywhere,
  title={Everywhere differentiability of infinity harmonic functions},
  author={Evans, Lawrence C and Smart, Charles K},
  journal={Calculus of Variations and Partial Differential Equations},
  volume={42},
  number={1},
  pages={289--299},
  year={2011},
  publisher={Springer}
}

@article{crandall2001optimal,
  title={Optimal Lipschitz extensions and the infinity Laplacian},
  author={Crandall, Michael G and Evans, Lawrence C and Gariepy, Ronald F},
  journal={Calculus of Variations and Partial Differential Equations},
  volume={13},
  number={2},
  pages={123--139},
  year={2001},
  publisher={Springer}
}

@article{crandall2009derivation,
  title={Derivation of the Aronsson equation for $C^1$ Hamiltonians},
  author={Crandall, Michael and Wang, Changyou and Yu, Yifeng},
  journal={Transactions of the American Mathematical Society},
  volume={361},
  number={1},
  pages={103--124},
  year={2009}
}

@article{aronsson1969minimization,
  title={Minimization problems for the functional $\sup_x F(x,f(x),f'(x))$},
  author={Aronsson, Gunnar},
  journal={Arkiv f{\"o}r matematik},
  volume={7},
  number={6},
  pages={509--512},
  year={1969},
  publisher={Springer}
}

@article{aronsson1967extension,
  title={Extension of functions satisfying Lipschitz conditions},
  author={Aronsson, Gunnar},
  journal={Arkiv f{\"o}r matematik},
  volume={6},
  number={6},
  pages={551--561},
  year={1967},
  publisher={Springer}
}

@article{aronsson1968partial,
  title={On the partial differential equation ux 2 u xx+ 2 uxuyu xy+ uy 2 u yy= 0},
  author={Aronsson, Gunnar},
  journal={Arkiv f{\"o}r matematik},
  volume={7},
  number={5},
  pages={395--425},
  year={1968},
  publisher={Springer}
}

@article{aronsson1984certain,
  title={On certain singular solutions of the partial differential equation ux2uxx+ 2uxuyuxy+ uy2uyy= 0},
  author={Aronsson, Gunnar},
  journal={Manuscripta mathematica},
  volume={47},
  number={1},
  pages={133--151},
  year={1984},
  publisher={Springer}
}

@article{evans2008c1_alpha,
  title={C 1, $\alpha$ regularity for infinity harmonic functions in two dimensions},
  author={Evans, Lawrence C and Savin, Ovidiu},
  journal={Calculus of Variations and Partial Differential Equations},
  volume={32},
  number={3},
  pages={325--347},
  year={2008},
  publisher={Springer}
}

@article{savin2005c,
  title={C 1Regularity for Infinity Harmonic Functions in Two Dimensions},
  author={Savin, Ovidiu},
  journal={Archive for rational mechanics and analysis},
  volume={176},
  number={3},
  pages={351--361},
  year={2005},
  publisher={Springer-Verlag, Berlin/Heidelberg}
}

@article{barron2008infinity_generalization,
  title={The infinity Laplacian, Aronsson’s equation and their generalizations},
  author={Barron, E and Evans, L and Jensen, R},
  journal={Transactions of the American Mathematical Society},
  volume={360},
  number={1},
  pages={77--101},
  year={2008}
}

@article{crandall2003efficient,
  title={An efficient derivation of the Aronsson equation},
  author={Crandall, Michael G},
  journal={Archive for rational mechanics and analysis},
  volume={167},
  number={4},
  pages={271--279},
  year={2003},
  publisher={Springer}
}

@article{everywhere_arrosson_peng2021,
  title={Everywhere differentiability of absolute minimizers for locally strongly convex Hamiltonian $H(p)\in C^{1,1}(\mathbb{R}^n)$ with $n\ge 3$},
  author={Peng, Fa and Miao, Qianyun and Zhou, Yuan},
  journal={Journal of Functional Analysis},
  volume={280},
  number={3},
  pages={108829},
  year={2021},
  publisher={Elsevier}
}

@inproceedings{c_1_arrosson_wang2008c,
  title={$C^1$-regularity of the Aronsson equation in $\mathbb{R}^2 $},
  author={Wang, Changyou and Yu, Yifeng},
  booktitle={Annales de l'IHP Analyse non lin{\'e}aire},
  volume={25},
  pages={659--678},
  year={2008}
}

@article{x_dependent_uniqueness_miao2017,
  title={Uniqueness of absolute minimizers for $L^{\infty}$-functionals involving Hamiltonians $H(x,p)$},
  author={Miao, Qianyun and Wang, Changyou and Zhou, Yuan},
  journal={Archive for Rational Mechanics and Analysis},
  volume={223},
  number={1},
  pages={141--198},
  year={2017},
  publisher={Springer}
}

@article{armstrong2011convexity,
  title={Convexity criteria and uniqueness of absolutely minimizing functions},
  author={Armstrong, Scott N and Crandall, Michael G and Julin, Vesa and Smart, Charles K},
  journal={Archive for rational mechanics and analysis},
  volume={200},
  number={2},
  pages={405--443},
  year={2011},
  publisher={Springer}
}

@article{uniqueness_nonunique_Arrosson_jensen2008,
  title={Uniqueness and nonuniqueness of viscosity solutions to Aronsson’s equation},
  author={Jensen, Robert and Wang, Changyou and Yu, Yifeng},
  journal={Archive for rational mechanics and analysis},
  volume={190},
  number={2},
  pages={347--370},
  year={2008},
  publisher={Springer}
}

@article{koskela2014intrinsic,
  title={Intrinsic geometry and analysis of diffusion processes and $L^{\infty}$-variational problems},
  author={Koskela, Pekka and Shanmugalingam, Nageswari and Zhou, Yuan},
  journal={Archive for Rational Mechanics and Analysis},
  volume={214},
  number={1},
  pages={99--142},
  year={2014},
  publisher={Springer}
}

@article{fu_Hammond_Pete_2022stake,
  title={Stake-governed tug-of-war and the biased infinity Laplacian},
  author={Fu, Yujie and Hammond, Alan and Pete, G{\'a}bor},
  journal={arXiv preprint arXiv:2206.08300},
  year={2022}
}

@article{liu2021weighted_ev,
  title={A weighted eigenvalue problem of the biased infinity Laplacian},
  author={Liu, Fang and Yang, Xiao-Ping},
  journal={Nonlinearity},
  volume={34},
  number={2},
  pages={1197},
  year={2021},
  publisher={IOP Publishing}
}

@article{liu2018inhomogeneous,
  title={An inhomogeneous evolution equation involving the normalized infinity Laplacian with a transport term},
  author={Liu, Fang},
  journal={Communications on Pure and Applied Analysis},
  volume={17},
  number={6},
  pages={2395--2421},
  year={2018},
  publisher={Communications on Pure and Applied Analysis}
}

@article{liu2019parabolic,
  title={Parabolic biased infinity Laplacian equation related to the biased tug-of-war},
  author={Liu, Fang and Jiang, Feida},
  journal={Advanced Nonlinear Studies},
  volume={19},
  number={1},
  pages={89--112},
  year={2019},
  publisher={De Gruyter}
}

@article{peres2019biased_graph,
  title={Biased infinity Laplacian Boundary Problem on finite graphs},
  author={Peres, Yuval and Sunic, Zoran},
  journal={arXiv preprint arXiv:1912.13394},
  year={2019}
}

@article{liu2022regularity,
  title={Regularity of Viscosity Solutions of the Biased Infinity Laplacian Equation},
  author={Liu, Fang and Meng, Fei and Chen, Xiaoyan},
  journal={Analysis in Theory and Applications},
  year={2022}
}

@article{armstrong2011infinity_gradient,
  title={An infinity Laplace equation with gradient term and mixed boundary conditions},
  author={Armstrong, Scott and Smart, Charles and Somersille, Stephanie},
  journal={Proceedings of the American Mathematical Society},
  volume={139},
  number={5},
  pages={1763--1776},
  year={2011}
}

@book{blanc2019game,
  title={Game theory and partial differential equations},
  author={Blanc, Pablo and Rossi, Julio Daniel},
  volume={31},
  year={2019},
  publisher={Walter de Gruyter GmbH \& Co KG}
}

@article{bjorland2012nonlocal,
  title={Nonlocal tug-of-war and the infinity fractional Laplacian},
  author={Bjorland, Clayton and Caffarelli, Luis and Figalli, Alessio},
  journal={Communications on pure and applied mathematics},
  volume={65},
  number={3},
  pages={337--380},
  year={2012},
  publisher={Wiley Online Library}
}

\end{document}